\documentclass[10pt]{article}

\usepackage{amssymb,amsmath,amsthm,amsfonts}
\usepackage{graphicx}
\usepackage{comment}
\graphicspath{{./pics/}}
\usepackage[colorlinks,linktocpage,linkcolor=blue]{hyperref}
\usepackage{algorithm,algorithmic}
\usepackage{color}
\usepackage[small,bf]{caption}

\newtheorem{thm}{Theorem}[section]

\newtheorem{lem}{Lemma}[section]

\newtheorem{cor}{Corollary}[section]

\title{An Inverse Sturm-Liouville Problem with a Fractional Derivative}
\author{Bangti Jin\footnote{Department of Mathematics and Institute for Applied Mathematics and Computational
Science, Texas A\&M University, College Station, Texas 77843-3368, USA (btjin,rundell@math.tamu.edu)}
\and William Rundell\footnotemark[1]
}
\date{\today}

\begin{document}
\maketitle
\begin{abstract}
In this paper, we numerically investigate an inverse problem of recovering the potential
term in a fractional Sturm-Liouville problem from one spectrum. The qualitative behaviors
of the eigenvalues and eigenfunctions are discussed, and numerical reconstructions of the
potential with a Newton method from finite spectral data are presented. Surprisingly, it
allows very satisfactory reconstructions for both smooth and discontinuous potentials,
provided that the order $\alpha\in(1,2)$ of fractional derivative is sufficiently away
from $2$.
\end{abstract}

\section{Introduction}

We consider the Sturm-Liouville problem (SLP) for the fractional differential equation
\begin{equation}\label{eqn:slp}
  \left\{\begin{aligned}
    &- D_0^\alpha u + qu = \lambda u\quad 0<x<1,\\
    &u(0)=u(1)=0,
  \end{aligned}\right.
\end{equation}
where the (left-sided) Caputo fractional derivative $D_0^\alpha$ of order
$\alpha\in(1,2)$ is defined by
\begin{equation*}
   D_0^\alpha u(x) = \frac{1}{\Gamma(2-\alpha)}\int_0^x\frac{u''(t)}{(x-t)^{\alpha-1}}dt,
\end{equation*}
where $\Gamma(\cdot)$ refers to the standard Gamma function. We shall normalize the
eigenfunction $u$ by $u'(0)=1$. As the exponent $\alpha$ tends to $2$, problem
\eqref{eqn:slp} reduces to the classical SLP \cite[Thm.~2.1,
pp.~92]{KilbasSrivastavaTrujillo:2006}. The fractional SLP \eqref{eqn:slp} is of immense
interest for several reasons.

First, it arises naturally in the analysis of (spatially) fractional wave equations when
applying Fourier transform. Fractional wave equations are often used to faithfully
capture dynamical behaviors of amorphous materials, e.g., polymer and porous media; see
the comprehensive reviews \cite{MetzlerNonnenmacher:2002,RossikhinShitikova:2010} and
references therein. Hence, a better understanding of the SLP \eqref{eqn:slp} would shed
valuable insight into the underlying physics of these phenological models.

Second, diffusion is one of the most prominent transport mechanisms found in nature. At a
microscopic level, it is the result of the random motion of individual particles, and the
use of the Laplace operator in the canonical model rests on a Gaussian process assumption
on the random motion. Over the last twenty years a wide body of the literature has shown
that anomalous diffusion in which the mean square variance grows faster (superdiffusion)
or slower (subdiffusion) than in a Gaussian process offers a superior fit to experimental
data (see the reviews \cite{SchlesingerWestKlafter:1987,SchneiderWyss:1989,
BouchaudGeorges:1990,MeerschaertBensonBaumer:1999,MetzlerKlafter:2000}). This is
particularly true in materials with memory, e.g., viscoelastic materials. This anomalous
diffusion can be in either the spatial or temporal variables and is typically reflected
in a fractional order derivative. In either case, separation of variables leads to a
(fractional order) ordinary differential equation with a set of constants, i.e., the
eigenvalues.

Third, it serves as a natural departure from the classical SLP ($\alpha=2$), for which
there is a wealth of profound theoretical results for both the forward and inverse
problems (see \cite[Chap.~3]{ChadanColtonPaivrintaRundell:1997} for an overview). It is
clearly interesting to see how these results might be translated to the new scenario both
from a mathematical perspective and for what this tells us about the underlying physics.

Suppose that we wish to solve the inverse problem of recovering the potential $q(x)$ from one
spectrum $\{\lambda_k\}$. It is well known that for the classical SLP ($\alpha=2$), this
is in general insufficient. One spectrum completely determines the potential only if
additional a priori information is available, for example, if it is known to be symmetric
about the midpoint of the interval or is given on one half of the interval and has to be
determined on the other half. In case of a general potential and fixed boundary
conditions at the endpoints one requires additional information in the form of a second
sequence. This can be a second spectrum arising from different boundary conditions,
specifying both the eigenfunction and its derivative at an endpoint, or giving the energy
associated with each frequency - the so-called norming constants (see
\cite[Chap.~3]{ChadanColtonPaivrintaRundell:1997}). However there are many situations
where this additional knowledge cannot be obtained, and one natural question is whether
the same requirements also hold for the fractional case, \eqref{eqn:slp}.

A number of efficient reconstruction techniques, including an iterative method involving
an associated hyperbolic equation \cite{RundellSacks:1992}, variational methods
\cite{Rohrl:2005}, and linearization methods using an iterated Newton method based on the
coefficient-to-data map \cite{LowePilantRundell:1992}, have been developed for the
classical case.

The main result of this paper is strong evidence that the fractional SLP \eqref{eqn:slp}
has very different properties and extensive numerical results lead to the conclusion that
one spectrum completely determines a general potential $q$ for $1<\alpha<2$, which
contrasts sharply with the classical case of $\alpha=2$. In this paper we will also deal
with reconstruction methods and follow the Newton scheme as the method of
choice.

We note that despite the extensive literature on the forward problem on fractional
ordinary/partial differential equations (see
\cite{Djrbashian:1993,Podlubny:1999,KilbasSrivastavaTrujillo:2006,LinXu:2007,BrunnerLingYamamoto:2010,SakamotoYamamoto:2011}
for an incomplete list) and their diverse physical and engineering applications
\cite{MetzlerKlafter:2004}, the research on relevant inverse problems remains very
scarce. Recently, several works on unique identifiability and numerical methods for
inverse problems for time-fractional diffusion/wave equations
\cite{ChengNakagawaYamamotoYamazaki:2009,LiuYamamoto:2010, ZhengWei:2011} have appeared
and some unusual phenomena over the classical case has been observed.

The paper is structured as follows. In Section \ref{sec:slp} we discuss results about the
Sturm-Liouville theory for fractional differential operators. In particular, we develop
an asymptotic formula for the spectrum and indicate the main properties of the
eigenfunctions. It turns out that the spectral values are complex and so we have both a
real and imaginary parts to the eigenfunctions and it is this feature that allows the
identifiability from a single, albeit complex spectrum. In Section \ref{sec:islp} we
numerically investigate the inverse SLP with a Newton-type method. Numerical results for
both smooth and discontinuous potentials clearly illustrate the phenomenon of recovering
a general potential for \eqref{eqn:slp} from one spectrum. We conclude the paper with
some future research problems in Section \ref{sec:concl}.

\section{Sturm-Liouville theory}\label{sec:slp}
In this section, we discuss qualitative behaviors of the spectrum to the SLP
\eqref{eqn:slp}. We first discuss the case of a zero potential, which unlike the situation
with $\alpha=2$, is quite nontrivial, and then turn to the case of a nonzero potential.

\subsection{Differential operator $-D_0^\alpha$}
In this part, we collect known results about the spectra of the fractional differential
operators $-D_0^\alpha(1<\alpha<2)$. We shall use the two-parameter Mittag-Leffler
function $E_{\alpha,\beta}(z)$ defined by
\begin{equation*}
  E_{\alpha,\beta}(z)=\sum_{k=0}^\infty \frac{z^k}{\Gamma(\beta+\alpha k)}.
\end{equation*}

The two particular versions of Mittag-Leffler function relevant to problem
\eqref{eqn:slp} are $E_{\alpha,2}(z)$ and $E_{\alpha,\alpha}(z)$. The former occurs in
the form $xE_{\alpha,2}(-\lambda x^\alpha)$ as the eigenfunctions of \eqref{eqn:slp} with
$q=0$, while the second arises in the kernel of the one-sided Green's function for
\eqref{eqn:slp} with $q=0$. These functions will also appear later in the Jacobian matrix
coming from the linearized problem. Further, the Dirichlet eigenvalues of the operator
$D_0^\alpha$ coincide with zeros of $E_{\alpha,2}(z)$ \cite[Thm. 4]{Nahusev:1977}.

If $0<\alpha<2$ and $\mu\in (\frac{\alpha\pi}{2},\min(\pi,\alpha\pi))$, then the function
$E_{\alpha,\beta}(z)$ has the following exponential asymptotic expansion (see e.g.,
\cite[Thm. 1.3-4, pp.~5]{Djrbashian:1993}, \cite[eq. (1.8.27),
p.~43]{KilbasSrivastavaTrujillo:2006}): as $|z|\rightarrow\infty$
\begin{equation}\label{eqn:asympt}
  E_{\alpha,\beta}(z) = \left\{\begin{aligned} &\frac{1}{\alpha}z^{\frac{1-\beta}{\alpha}}\exp(z^\frac{1}{\alpha})-
     \sum_{k=1}^N\frac{1}{\Gamma(\beta-\alpha k)}\frac{1}{z^k}+O\left(\frac{1}{z^{N+1}}\right),  & |\arg(z)|\leq \mu,\\
     &-\sum_{k=1}^N\frac{1}{\Gamma(\beta-\alpha k)}\frac{1}{z^k}+O\left(\frac{1}{z^{N+1}}\right), & \mu\leq |\arg(z)|\leq\pi.
     \end{aligned}\right.
\end{equation}

The fractional SLP \eqref{eqn:slp} was studied earlier \cite{Djrbashian:1970,
Nahusev:1977}. In \cite{Djrbashian:1970}, the existence of a solution to such boundary
value problem was established, and a certain biorthogonal system was constructed from the
eigenfunctions, see also \cite[Chaps.~10-12]{Djrbashian:1993}. In \cite{Nahusev:1977},
the aforementioned relation between eigenvalues and zeros of Mittag-Leffler function was
shown. The distribution of the zeros of the function $E_{\alpha,\beta}(z)$ is of
independent interest \cite[Chap. 1]{Djrbashian:1993} \cite{Sedletskii:1994,Popov:2008}.
The next result \cite[Lemma~1.4-2, pp.~7]{Djrbashian:1993} represents one of the main
results. For the sake of completeness, we sketch the proof.
\begin{thm}\label{thm:asympt}
For all sufficiently large $n$ (in absolute value), the zeros $\{z_n\}$ of the
Mittag-Leffler function $E_{\alpha,2}(z)$ are simple, and have the following asymptotics
\begin{equation*}
   z_n^{\frac{1}{\alpha}} = 2n\pi \mathrm{i} - (\alpha-1)\left(\ln 2\pi |n| +\frac{\pi}{2}
   \mathrm{sign}(n)\,\mathrm{i}\right) + \ln \frac{\alpha}{\Gamma(2-\alpha)} + d_n,
\end{equation*}
where the remainder $d_n$ is $O\left(\frac{\ln |n|}{|n|}\right)$.
\end{thm}

\begin{proof}
By taking $N=1$ in the exponential asymptotics \eqref{eqn:asympt}, we get
\begin{equation*}
  E_{\alpha,2}(z) = \frac{1}{\alpha}z^{-\frac{1}{\alpha}}e^{z^\frac{1}{\alpha}}-\frac{1}{\Gamma(2-\alpha)}\frac{1}{z} + O\left(\frac{1}{z^2}\right)
\end{equation*}
for $|z|\rightarrow \infty$. Hence we have
\begin{equation*}
  z^{1-\frac{1}{\alpha}}e^{z^\frac{1}{\alpha}}=\frac{\alpha}{\Gamma(2-\alpha)} + O\left(\frac{1}{z}\right).
\end{equation*}
Next let $\zeta = z^\frac{1}{\alpha}$ and $w=\zeta + (\alpha-1)\ln\zeta$. Then the
equation can be rewritten as
\begin{equation*}
  e^w = \frac{\alpha}{\Gamma(2-\alpha)} + O\left(\frac{1}{w^\alpha}\right).
\end{equation*}
The solutions $w_n$ to the above equation for all sufficiently large $n$ satisfy
\begin{equation*}
   w_n = 2\pi n\, \mathrm{i} + \ln\frac{\alpha}{\Gamma(2-\alpha)} + O\left(\frac{1}{|n|^\alpha}\right),
\end{equation*}
or equivalently
\begin{equation*}
  \zeta_n+(\alpha-1)\ln\zeta_n = 2\pi n\, \mathrm{i} + \ln \frac{\alpha}{\Gamma(2-\alpha)} + O\left(\frac{1}{|n|^\alpha}\right),
\end{equation*}
from which we arrive at the desired assertion
\begin{equation*}
    \zeta_n = 2\pi n\, \mathrm{i} + \ln\frac{\alpha}{\Gamma(2-\alpha)} - (\alpha-1)
          \left(\ln 2\pi |n| + \frac{\pi}{2}\mathrm{sign}(n)\, \mathrm{i}\right) + O\left(\frac{\ln|n|}{|n|}\right).
\end{equation*}
\end{proof}
We can improve the estimate in Theorem \ref{thm:asympt} by including further terms in the
residual tail in the asymptotic \eqref{eqn:asympt}, $N\geq 2$, but with considerable
increases in the complexity of the result.

Clearly, if $\lambda\in\mathbb{C}$ is an eigenvalue of $-D_0^\alpha$, then its conjugate
$\overline{\lambda}$ is also an eigenvalue. Hence we shall restrict our discussions to
$n>0$. As a direct consequence of Theorem \ref{thm:asympt}, there are only finitely many real
eigenvalues for any $\alpha<2$. The existence of real eigenvalues is only guaranteed for
$\alpha$ sufficiently close to $2$ \cite{Popov:2008}. Further, asymptotically, the
eigenvalues $\{\lambda_n\}$ are distributed as
\begin{equation*}
  (-\lambda_n)^\frac{1}{\alpha} \sim 2\pi n\, \mathrm{i} + \ln\frac{\alpha}{\Gamma(2-\alpha)} -
     (\alpha-1) \left(\ln 2\pi n + \frac{\pi}{2}\,\mathrm{i}\right).
\end{equation*}

Hence, we immediately get the following corollary.

\begin{cor}
Asymptotically, the magnitude $|\lambda_n|$ and phase $\arg(\lambda_n)$ of the
$n^\mathrm{th}$ eigenvalue $\lambda_n$ are given by
\begin{equation}\label{eqn:eigasym}
  \begin{aligned}
   |\lambda_n| & \sim \left((2\pi n + (1-\alpha)\tfrac{\pi}{2})^2+ ((1-\alpha)\ln 2\pi n +
       \ln \tfrac{\alpha}{\Gamma(2-\alpha)})^2\right)^\frac{\alpha}{2}\sim (2\pi n)^\alpha,\\
   \arg(\lambda_n)&  \sim \pi- \alpha\mathrm{atan}\frac{2\pi n + (1-\alpha)\tfrac{\pi}{2}}{(\alpha-1)\ln 2\pi n
      - \ln \tfrac{\alpha}{\Gamma(2-\alpha)}}\sim \frac{(2-\alpha)\pi}{2}.
  \end{aligned}
\end{equation}
\end{cor}

\begin{figure}[h!]
  \centering
  \begin{tabular}{cc}
  \includegraphics[trim = 20mm 2mm 20mm 2mm, clip, width=.5\textwidth]{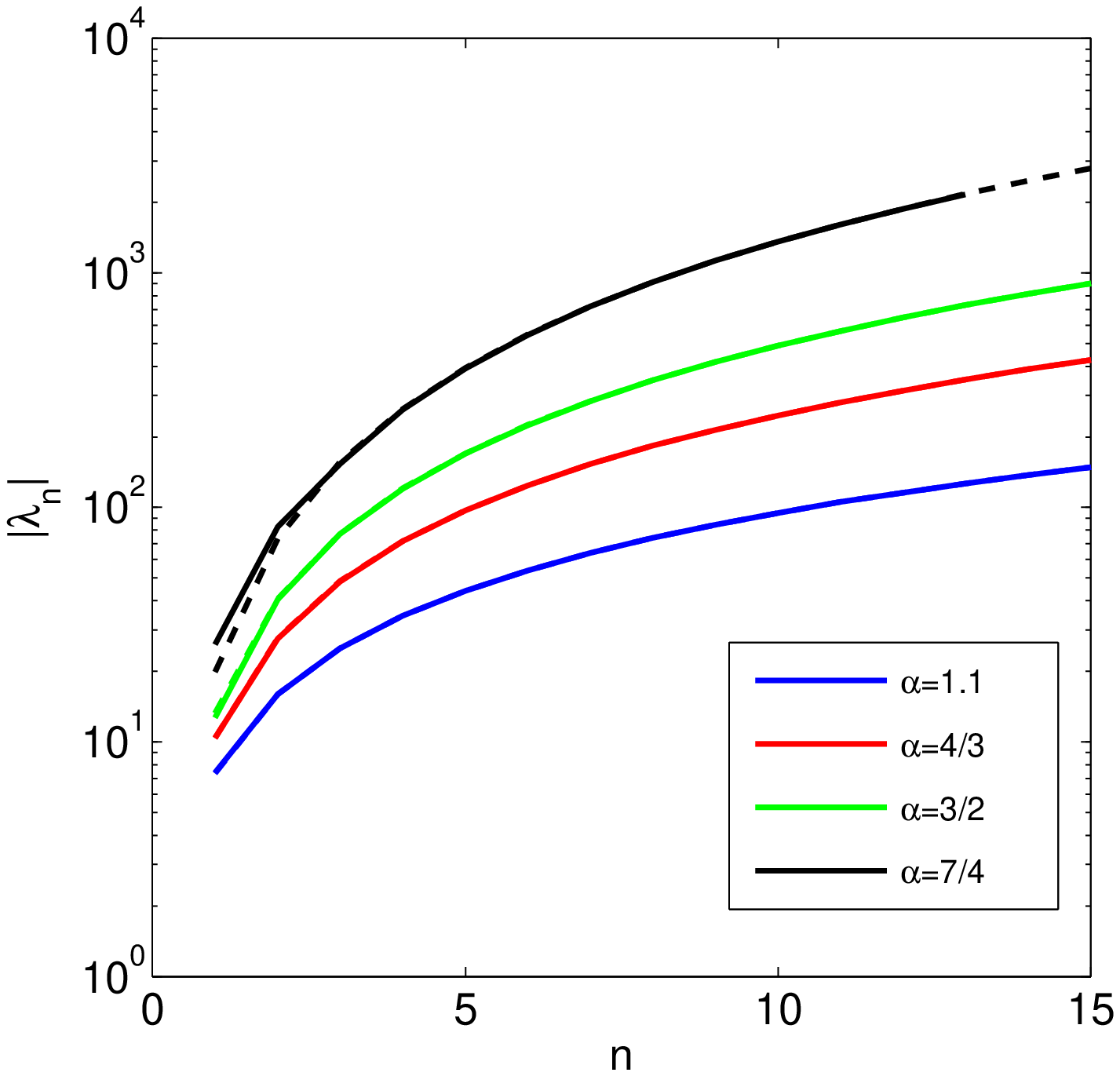} &
  \includegraphics[trim = 20mm 2mm 20mm 2mm, clip, width=.5\textwidth]{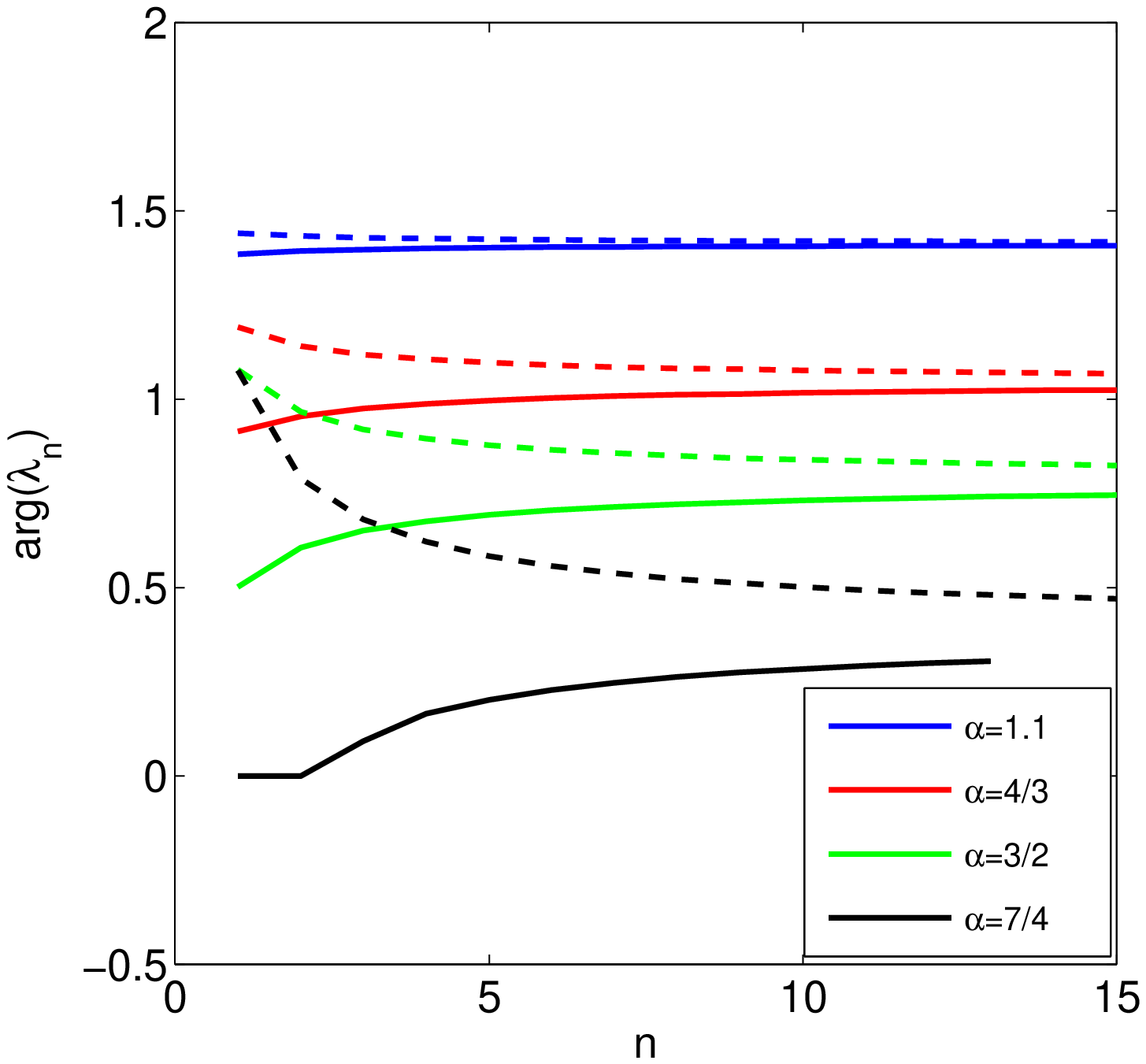}\\
  (a) magnitude $|\lambda_n|$ & (b) phase $\mathrm{arg}(\lambda_n)$
  \end{tabular}
  \parbox{5in}{\caption{Numerical verification of the asymptotic formula \eqref{eqn:eigasym}:
  the solid and dashed lines represent the true and predicted values, respectively, for
  $\alpha=1.1$, $4/3$, $3/2$ and $7/4$.}\label{fig:lam}}
\end{figure}

We show in Figure~\ref{fig:lam} the magnitude and phase of true eigenvalues and their
predictions by the asymptotic formula \eqref{eqn:eigasym}. The true eigenvalues are
calculated by a quasi-Newton method (cf. Appendix \ref{sec:newton}), and the values are
precise in the first six digits. The magnitudes can be accurately predicted by formula
\eqref{eqn:eigasym}, except the first few eigenvalues. Hence, it might allow extracting
the exponent $\alpha$ directly from the sequence of eigenvalues. The approximation of the
phase is not as good, cf. Figure~\ref{fig:lam}(b), and it is accurate only for very large
$n$, especially for $\alpha$ values close to $2$. This is attributed to both crude
approximation in deriving the asymptotic formula \eqref{eqn:eigasym} and the slow
convergence of the function \texttt{atan} to $\frac{\pi}{2}$. This is in stark contrast
with the classical case $\alpha=2$, where the asymptotic formula is very accurate even
for relatively small $n$ \cite{ChadanColtonPaivrintaRundell:1997}. The first ten
eigenvalues are shown in Figure~\ref{fig:lamval} for $\alpha=1.10$ and $1.75$. We observe
that, at least for $\alpha$ not close to unity, the asymptotic formula
\eqref{eqn:eigasym} is a poor approximation for small spectral numbers. So one cannot
make much use of this from finite spectral data consisting of the lowest frequencies and
it is precisely such data that is most easily measured in practice.

\begin{figure}
  \centering
  \begin{tabular}{cc}
  \includegraphics[trim = 20mm 2mm 20mm 2mm, clip, width=.5\textwidth]{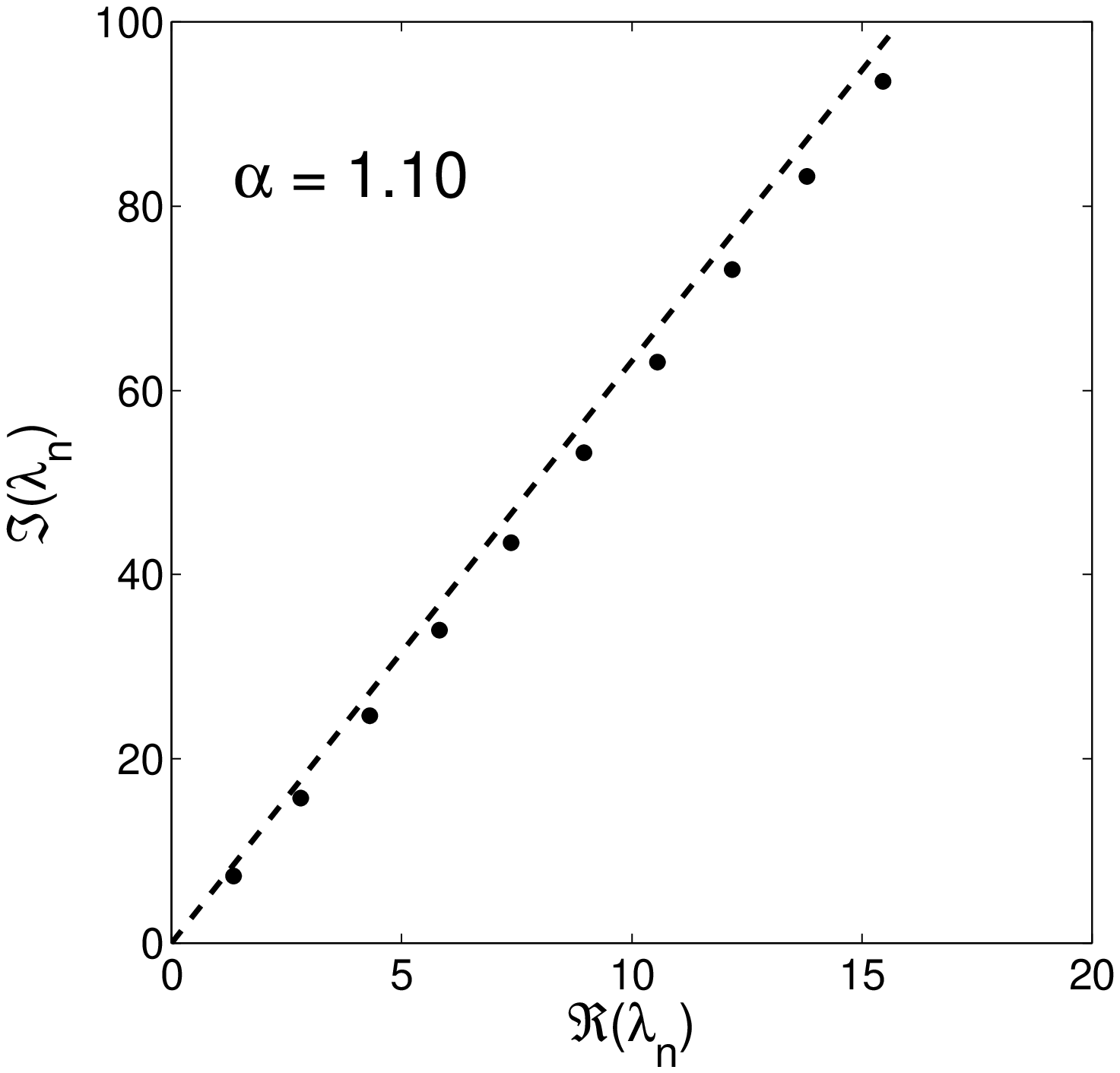} &
  \includegraphics[trim = 20mm 2mm 20mm 2mm, clip, width=.5\textwidth]{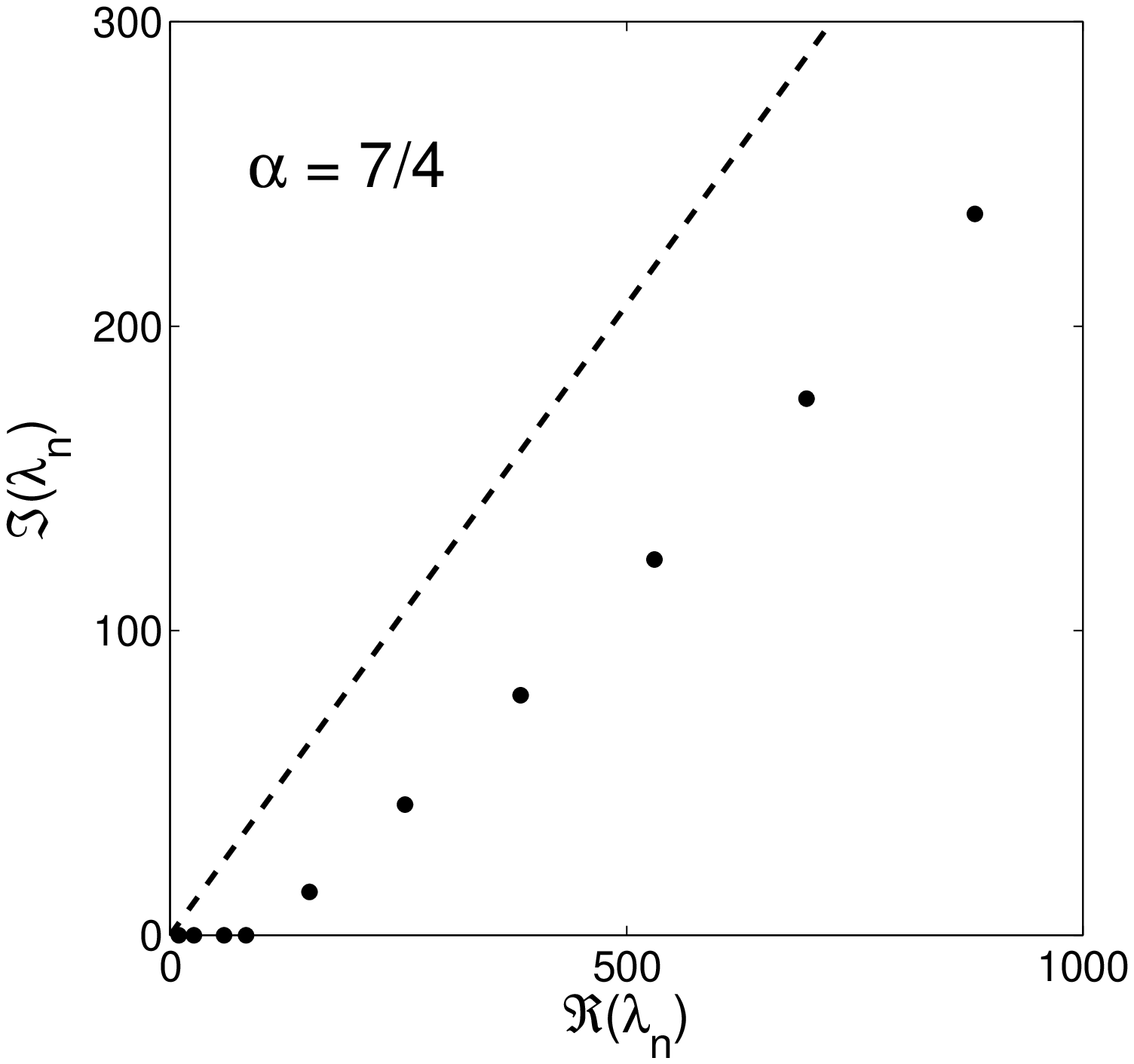}
  \end{tabular}
  \parbox{5in}{\caption{The first ten eigenvalues for $\alpha=1.10$ (left) and $\alpha=1.75$ (right). Here the dashed
  line represents the asymptotic with angle $(1-\frac{\alpha}{2})\pi$.}\label{fig:lamval}}
\end{figure}

The eigenfunctions $u_n$ for the operator $-D_0^\alpha$ are given by $ u_n =
xE_{\alpha,2}(-\lambda_nx^\alpha)$. These functions are reminiscent of sinusoids (cf.
Figure~\ref{fig:eigfcn}), the eigenfunctions for $\alpha=2$, but significantly attenuated
close to $x=1$. The degree of attenuation strongly depends on fractional order $\alpha$.
It is observed that for a complex eigenvalue, the number of interior zeros of the real
and imaginary parts of the respective eigenfunction $u_n$ always differs by one, and that
(either real or imaginary part) for two consecutive eigenvalues differs by two. However,
the number of zeros for the consecutive real eigenvalues differ only by one (see the last
row of Figure~\ref{fig:eigfcn}) just as for the classical SLP.

\begin{figure}
  \centering
  \begin{tabular}{cc}
   \includegraphics[trim = 20mm 2mm 20mm 2mm, clip, width=.5\textwidth]{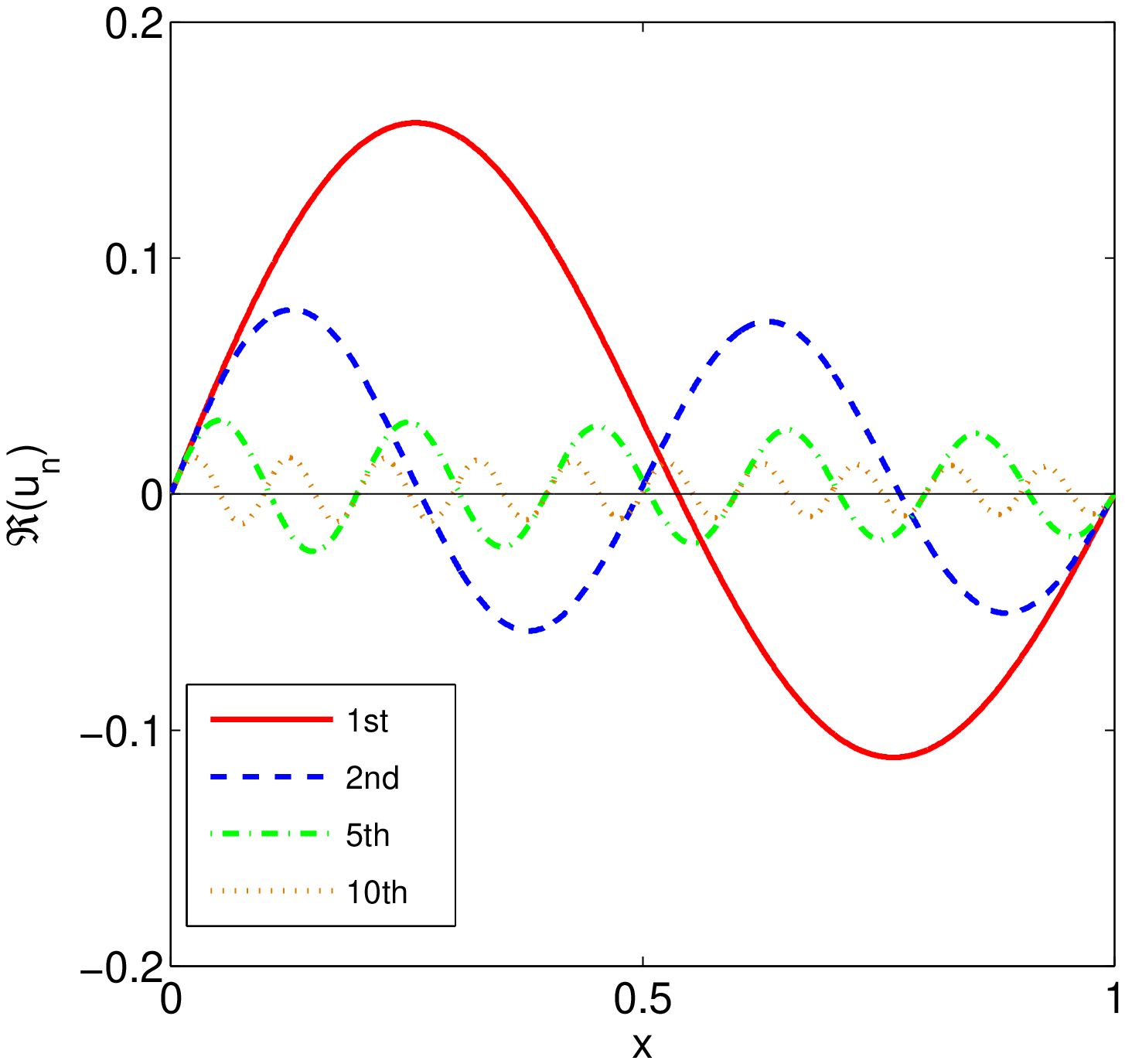} &
   \includegraphics[trim = 20mm 2mm 20mm 2mm, clip, width=.5\textwidth]{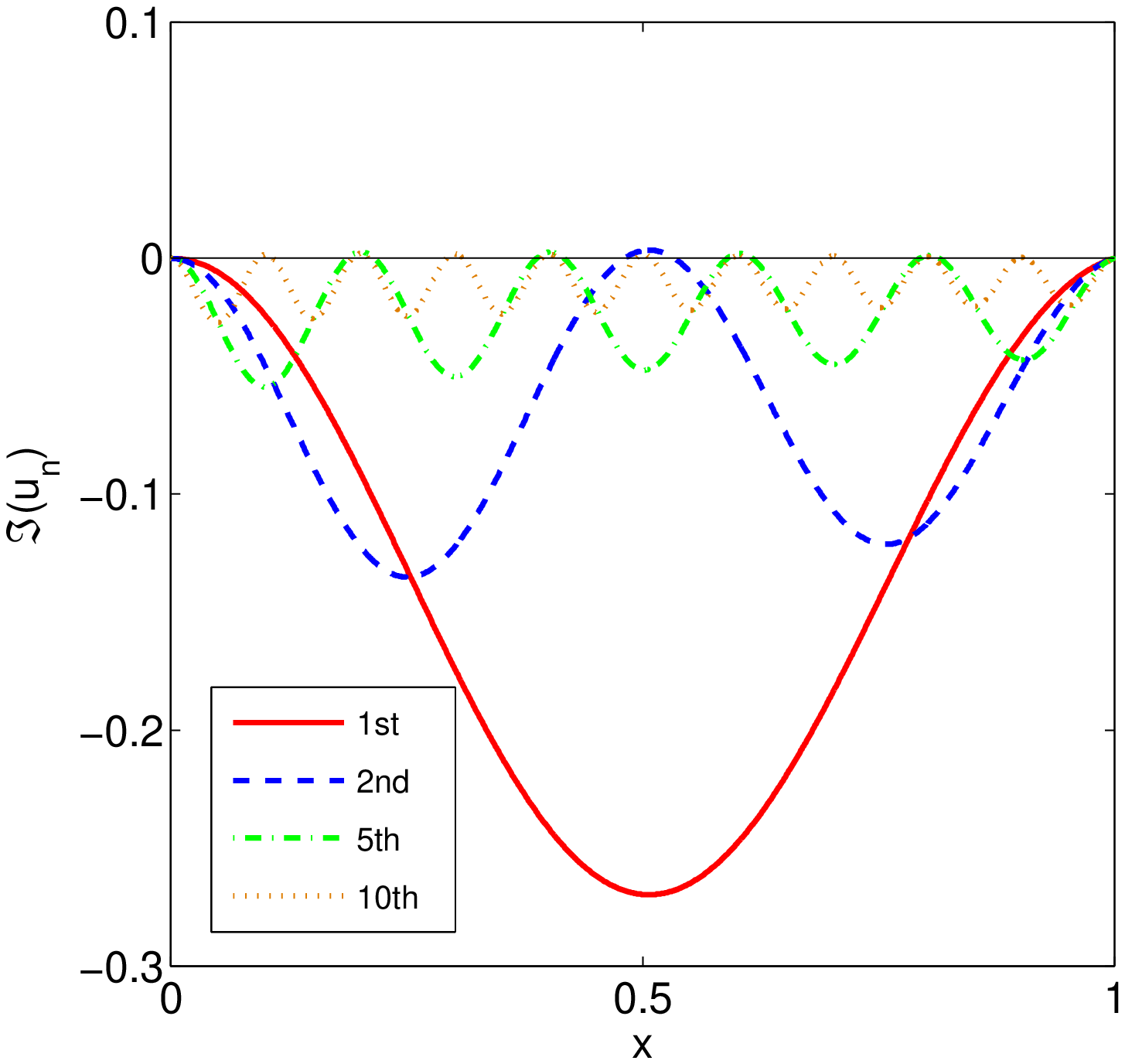}\\
   \includegraphics[trim = 20mm 2mm 20mm 2mm, clip, width=.5\textwidth]{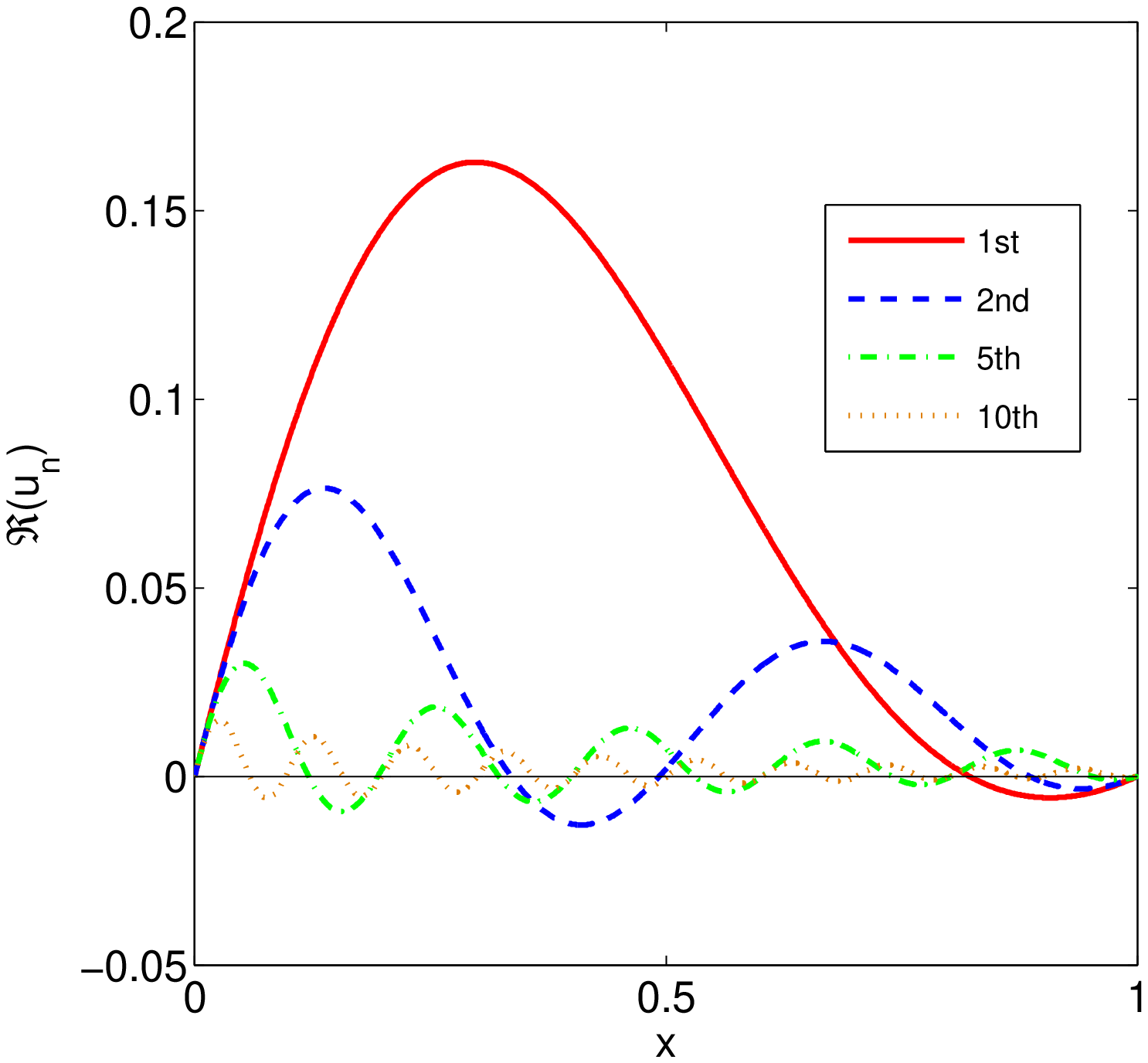} &
   \includegraphics[trim = 20mm 2mm 20mm 2mm, clip, width=.5\textwidth]{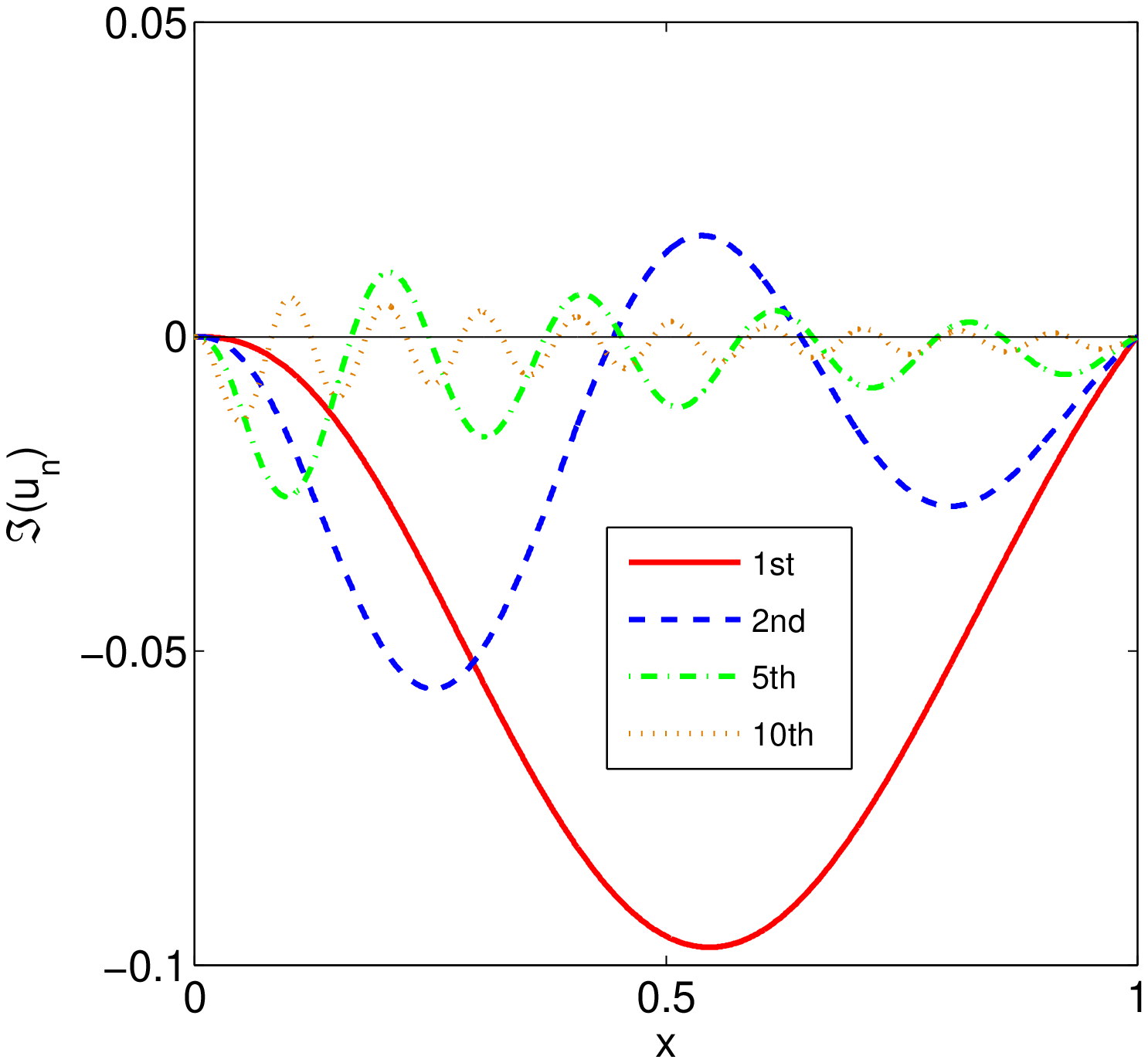}\\
   \includegraphics[trim = 20mm 2mm 20mm 2mm, clip, width=.5\textwidth]{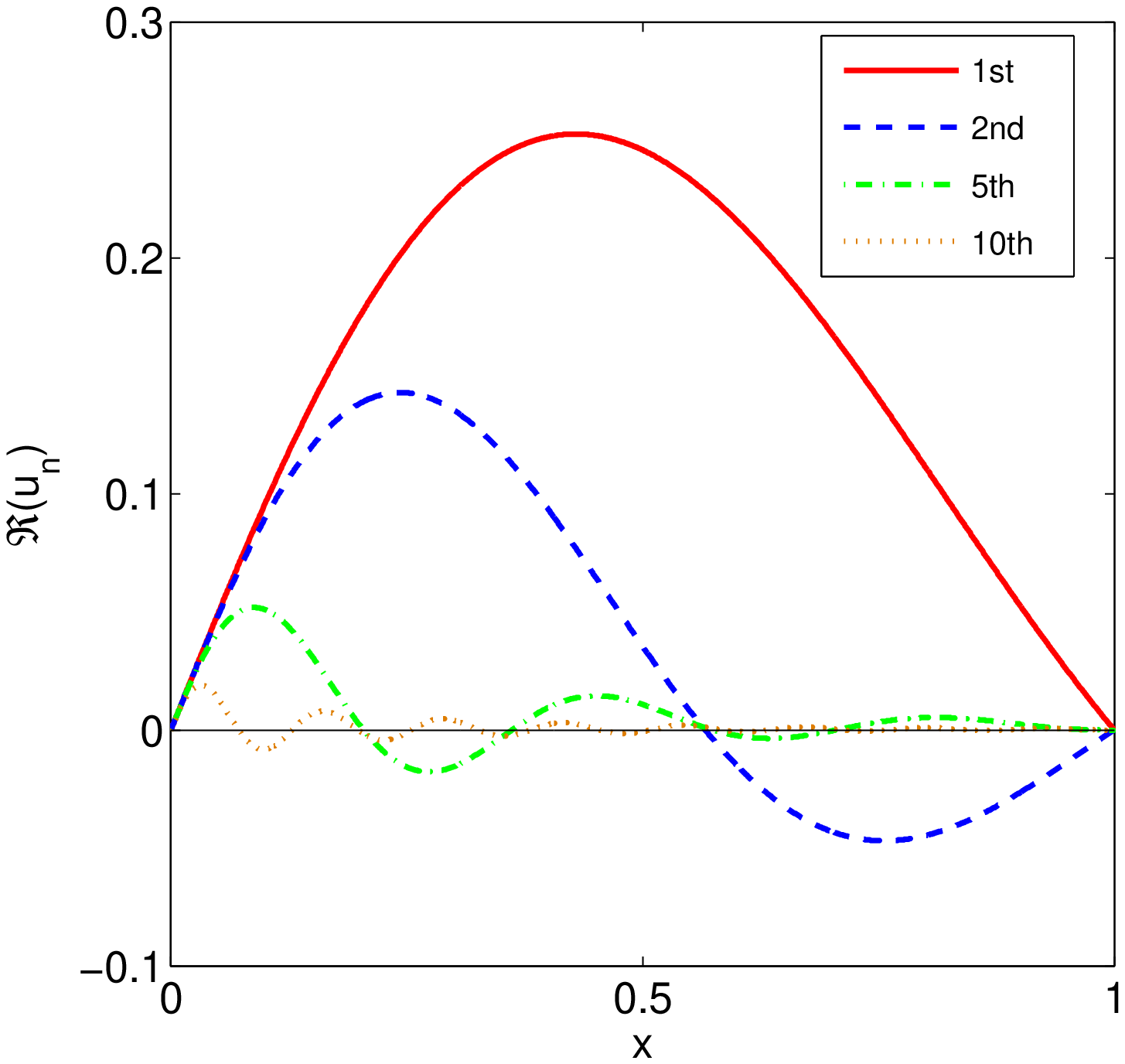} &
   \includegraphics[trim = 20mm 2mm 20mm 2mm, clip, width=.5\textwidth]{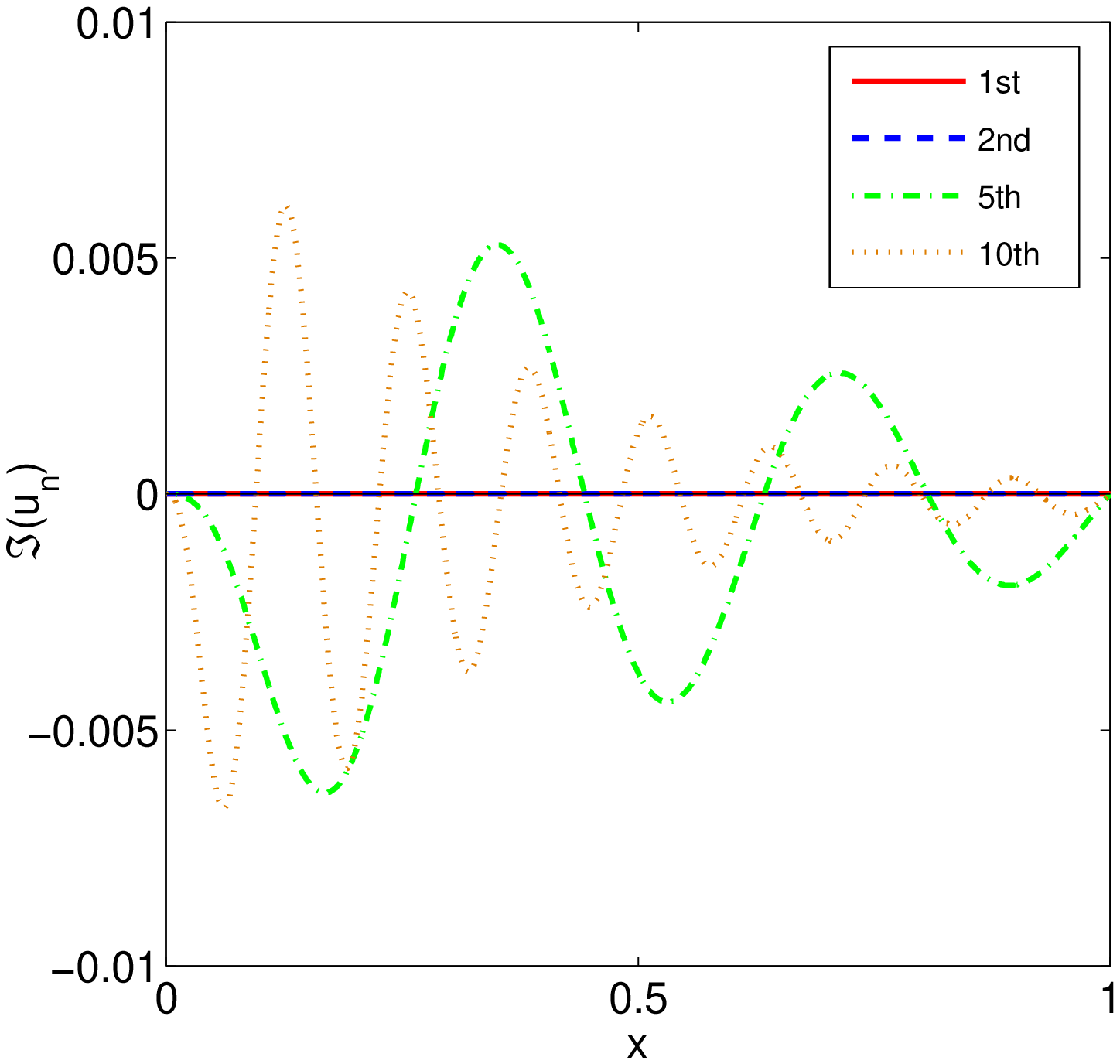}
  \end{tabular}
  \parbox{5in}{\caption{Real (left) and imaginary (right) parts of eigenfunctions $u_n$ for
  $\alpha=1.10,1.50$ and $1.75$ (from top to bottom). }\label{fig:eigfcn}}
\end{figure}

In the classical inverse SLP almost all reconstruction algorithms make use, directly or
indirectly, of the fact that the eigenvalues are simple and, moreover, for potentials
adjusted to have mean zero, there is an interval condition: there is exactly one
(Dirichlet) eigenvalue in each interval $(n\pi,(n+1)\pi)$ \cite[Chap.
3]{ChadanColtonPaivrintaRundell:1997}. In fact a similar result can be shown for the
function $E_{\alpha,\beta}$ with $\alpha=2$ and $1<\beta<3$ \cite[Thm.
1.4-2]{Djrbashian:1993}. However, if $1<\alpha<2$ then this extremely useful property is
either no longer ensured, or indeed does not hold. The asymptotic formula shows that all
sufficiently large eigenvalues are simple, but it remains an open question if this is
true in general -- even for the case $q=0$!

In the classical case, the eigenfunctions have the asymptotic form $u_n(x) =
\sin(\sqrt{\lambda_n}x)/\sqrt{\lambda_n}$ ($\cos(\sqrt{\lambda_n}x)$ in the case of
non-Dirichlet conditions at $x=0$) \cite{ChadanColtonPaivrintaRundell:1997}, and these
are very good approximations for moderately sized smooth potentials. Thus distinguishing
eigenfunctions as well as eigenvalues is straightforward. Indeed, this clear
correspondence also carries over to the potential recovery: if we are missing information
about the $m^{\rm th}$ eigenvalue, then we are unable to obtain information about the
$m^{\rm th}$ Fourier mode of the potential \cite{Hochstadt:1973} but the other modes are
essentially unaffected.

In the case $\alpha<2$, the picture is quite different.
It is known that for $\alpha>5/3$ there are at least two real zeros (and hence
eigenvalues) to the function $E_{\alpha,2}(z)$ \cite[Thm. 1]{Popov:2008}. Careful
numerical experiments indicate that the first real zeros appear around $\alpha=1.6$ (more
precisely within the interval $(1.5991152,1.5991153)$) and they occur in pairs (there are
4 real eigenvalues for $\alpha=1.75$ as Figure~\ref{fig:lamval} shows) so there appears
to be always an even number of these. However these pairs can be quite close to each
other and so there is no possible interval condition. For example, with $\alpha =
1.5991153$, the two smallest eigenvalues in magnitude are real with approximate values
$14.0024$ and $14.0150$. Further refinement indicates these two eigenvalues are genuinely
simple. The eigenfunction of the lower eigenvalue has no zeros in $(0,1)$, while the
larger has one single zero and so they are linearly independent. However, the zero occurs
at $x\approx 0.9994$ and the supremum norm difference of the two is less than
$1.834\times10^{-4}$. Thus for all practical computational purposes, neglecting one of
these pairs will have no effect. With $\alpha=1.599025$, there are no real roots and the
two smallest eigenvalues in magnitude are the complex conjugate pairs $14.0062\pm
0.1955\,\mathrm{i}$. As we noted above, the imaginary part of the corresponding
eigenfunction has no zeros in $(0,1)$, and the real part does indeed have a zero in this
interval. However, this is at a point $x_0\in (0.9997,0.9998)$ and the maximum value of
the real part on $(x_0,1)$ is less that $5.2501\times10^{-9}$. Thus practically speaking,
this zero plays no computational role. These examples also highlight the difficulty in
obtaining analytic results or estimates on the properties of the eigenvalues or
eigenfunctions.
\subsection{Differential operator $-D_0^\alpha+q(x)$}

In the classical SLP, the presence of the potential $q(x)$ influences the eigenvalues by
\begin{equation}\label{eqn:decay}
  \lambda_n (q) = \lambda_n(0) + \int_0^1q(t)dt + c_n,
\end{equation}
where the remainders $c_n$ decay to zero as $n\rightarrow+\infty$. The sequence $\{c_n\}$
effectively encodes all the information about the potential $q(x)$. This formula
represents only a qualitative behavior. In practice, the decay rate of $c_n$ can differ
markedly, dependent on the smoothness of the potential: the smoother is the potential,
the faster is the decay. Numerically, we can observe similar behavior for fractional
SLPs.
\begin{figure}[h!]
  \centering
  \begin{tabular}{cc}
     \includegraphics[trim = 20mm 2mm 20mm 2mm, clip, width=.5\textwidth]{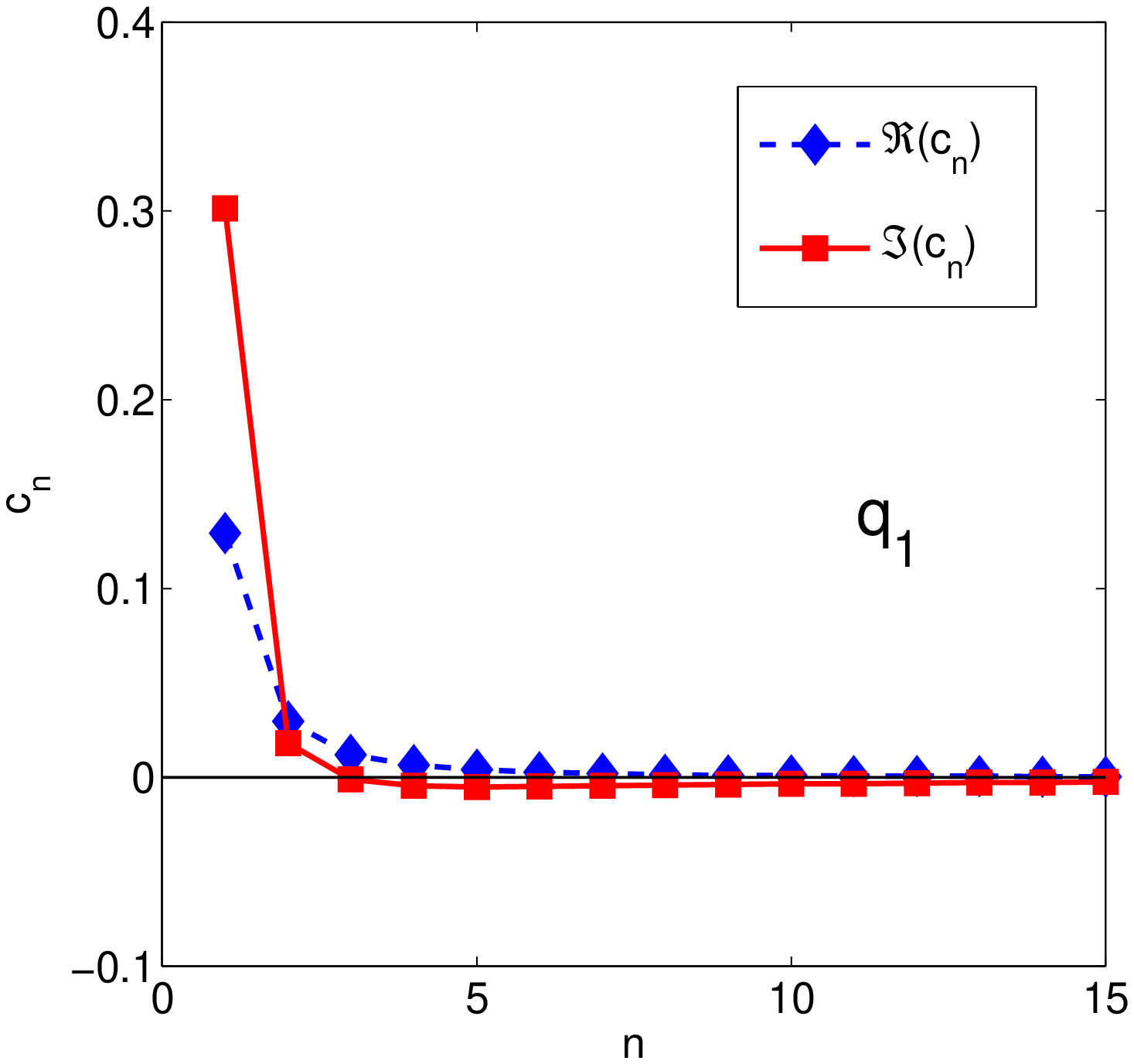} &
     \includegraphics[trim = 20mm 2mm 20mm 2mm, clip, width=.5\textwidth]{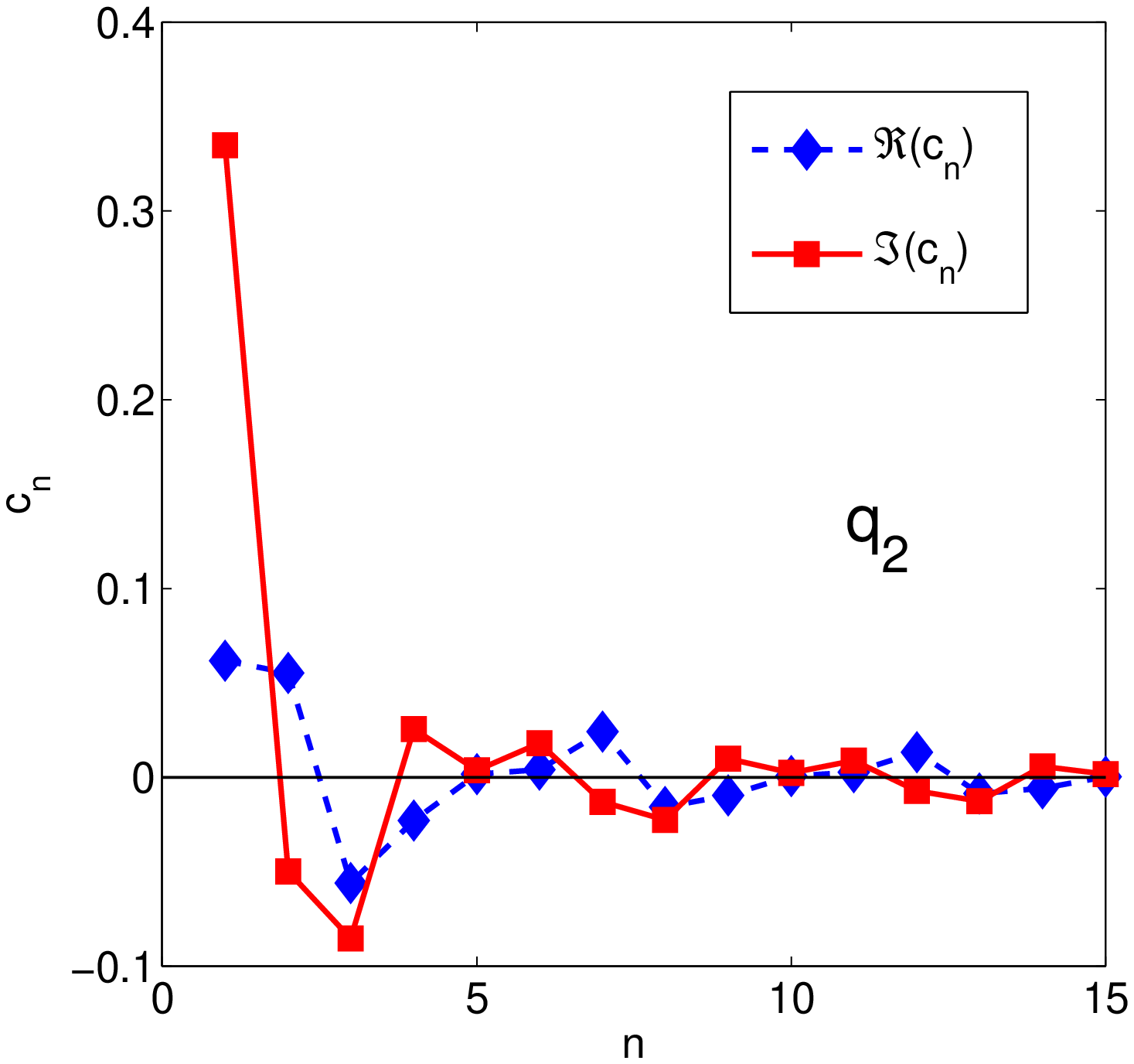}
  \end{tabular}
  \parbox{5in}{\caption{The decay of the remainder $c_n$ for $\alpha=3/2$.}\label{fig:eigdc}}
\end{figure}

\noindent To illustrate the point, we compute the eigenvalues for the following two
potentials
\begin{equation*}
    q_1(x) = 20x^3(e^{-(x-\frac{1}{2})^2}-e^{-\frac{1}{4}})\quad \mbox{and}\quad
     q_2(x) =\left\{\begin{array}{ll}
            -2x, &  x\in [0,\tfrac{1}{5}],\\
            -\tfrac{4}{5}+2x, & x\in[\tfrac{1}{5},\tfrac{2}{5}],\\
            1, &x\in[\tfrac{3}{5}, \tfrac{4}{5}],\\
            0, & \mbox{elsewhere}.
        \end{array}\right.
\end{equation*}
The profiles of the potentials can be found in Figure~\ref{fig:recon} shown later in \S
\ref{ssec:recon}. The differences between the eigenvalues for the differential operator
$-D_0^\alpha+q(x)$ and that of $-D_0^\alpha +\int_0^1q(t)dt$ are shown in
Figure~\ref{fig:eigdc} for $\alpha=3/2$. The decay of the remainders $c_n$ is much faster
for the smooth potential $q_1$, which holds for both real and imaginary parts of
eigenvalues. Hence, the first few eigenvalues might allow a reasonable reconstruction of
the potential in the inversion procedures. In contrast, in case of the discontinuous
potential $q_2$, the remainder $c_n$ still has noticeably large oscillations for $n$ up
to $15$. These observations are also valid for other $\alpha$ values, and concur with
that for the classical SLP \cite[Figure~1, pp.~82]{ChadanColtonPaivrintaRundell:1997}. We
note that the presence of a nonzero potential may render two neighboring real eigenvalues
into a pair of complex eigenvalues. For example, in case of $\alpha=1.6$, there are two
real eigenvalues for $q=0$, i.e., $13.4205$ and $14.6454$, whereas that for the
potentials $q_1$ and $q_2$ are given by the complex conjugate pair
$14.6843\pm1.7197\mathrm{i}$ and $ 14.2242\pm1.7910\mathrm{i}$, respectively.

\section{An inverse Sturm-Liouville problem}\label{sec:islp}
Now we turn to the inverse problem of recovering the potential $q(x)$ in the fractional
SLP \eqref{eqn:slp} from (finite) spectral data $\{\lambda_n\}_{n=1}^N$, with the help of a
simplified Newton method. Numerically we observe that one single spectrum can uniquely
determine a general potential.

\subsection{A Newton method}
We shall numerically solve the inverse SLP by a Newton type method based on that
of \cite{LowePilantRundell:1992} for the classical case.
First, we denote by $u(q,\lambda)$ the solution to the
initial value problem
\begin{equation}\label{eqn:ivp}
  \left\{\begin{aligned}
   &-D_0^\alpha u + qu = \lambda u \quad 0<x<1,\\
   & u (0 )= 0,\quad u'(0) = 1.
  \end{aligned}\right.
\end{equation}
Obviously, the number $\lambda\in\mathbb{C}$ will be an eigenvalue if the solution
$u(q,\lambda)$ satisfies $u(q,\lambda)(1)=0$.

The knowledge of only finite spectral data $\{\lambda_n\}_{n=1}^N$ or equivalently the vector
$\Lambda_N=(\lambda_1,\ldots,\lambda_N)^\mathrm{t}\in\mathbb{C}^N$, as is often the case
in practice, precludes the determination of an arbitrary potential $q$.
Instead, we seek a potential in a finite-dimensional space spanned
by the basis $\{w_k\}$, and we take its dimension to be $M\leq N$,
allowing us the possibility to reduce this value for purposes of regularization, i.e.,
\begin{equation*}
  q^M(x) = \sum_{k=1}^M q_kw_k(x).
\end{equation*}

Now the problem is to find $q^M\in\mathrm{span}(\{w_k\})$ such that $u(q^M,\lambda)(1)$
vanishes, i.e., $u$ satisfies also the desired right-side boundary condition. We define a
nonlinear mapping $F:\mathbb{R}^M\mapsto\mathbb{C}^N$ by
\begin{equation*}
  F_n(\Lambda_N,q^M) = u(q^M,\lambda_n)(1)\quad n = 1,\ldots,N.
\end{equation*}
For the given set $\Lambda_N$ of eigenvalues, we attempt to solve the system of nonlinear
equations
\begin{equation*}
  F(\Lambda_N,q^M) = 0.
\end{equation*}

Thus Newton's method emerges as a natural candidate (see \cite{LowePilantRundell:1992}
for the case $\alpha=2$) and we can proceed as follows: given an initial guess
$q^0\in\mathrm{span}(\{w_k\}_{k=1}^M)$, we repeat the iteration
\begin{equation}\label{eqn:newton}
  \mathbf{q}^{n+1} = \mathbf{q}^n -(F'(\Lambda_N,q^n))^{-1}(F(\Lambda_N,q^n))\quad n=0,1,\ldots.
\end{equation}
where the vector $\mathbf{q}^n=(q^n_1,\ldots,q^n_M)\in\mathbb{R}^M$ refers to the
expansion coefficients of the approximate potential $q^n$ in the basis $\{w_k\}$, until a certain
stopping criterion is reached. The entries of the Jacobian $F'(\Lambda_N,q)$ can be found
by solving a set of fractional differential equations as we show in the next lemma.
\begin{lem}
The $nk^{th}$ entry of the Jacobian $F'(\Lambda_N,q)$, i.e., $\frac{\partial
F_n(\Lambda_N,q)}{\partial q_k}$, is given by $v(\lambda_n,q,w_k)(1)$, where the function
$v(\lambda_n,q,w_k)$ satisfies
\begin{equation}\label{eqn:jac}
  \left\{  \begin{aligned}
     &-D_0^\alpha v + q v = \lambda_n v - w_k u (\lambda_n,q)\quad 0<x<1,\\
     & v(0) = 0, \ \ v'(0)= 0.
  \end{aligned}\right.
\end{equation}
\end{lem}

Therefore, evaluating either the map $F(\Lambda_N,q)$ or the Jacobian $F'(\Lambda_N,q)$
incurs solving initial value problems of fractional order (cf. \eqref{eqn:ivp} and
\eqref{eqn:jac}), which can be numerically accomplished by the predictor-corrector method
in Appendix \ref{sec:app}. The main computational effort in applying the Newton method
\eqref{eqn:newton} lies in computing the Jacobian $F'(\Lambda_N,q)$. We note that for
large $\lambda$ values, the initial value problem is very stiff, and a very refined mesh
may be needed to resolve the problem to a prescribed accuracy.

Of particular interest is the special case $q=0$ and using the respective eigenvalues
$\Lambda_{N,0}$. Then the solution $v(\lambda_{n,0},0,w_k)$ to problem \eqref{eqn:jac}
can be expressed in closed form using the Green's function of the operator
$D_0^\alpha-\lambda$ \cite[eq. (4.1.74), p.~232]{KilbasSrivastavaTrujillo:2006}
\begin{equation*}
  v(\lambda_{n,0},0,w_k)(x) = \int_0^x (x-t)^{\alpha-1}E_{\alpha,\alpha}(-\lambda_{n,0}(x-t)^\alpha) tE_{\alpha,2}(-\lambda_{n,0}t^\alpha)w_k(t)dt.
\end{equation*}

With the approximation $F'(\Lambda_{N,0},0)$ in place of the Jacobian
$F'(\Lambda_N,q^n)$, the iteration \eqref{eqn:newton} leads to the canonical frozen
Newton's method, which is known to converge (locally) linearly. This reduces the
computational expense enormously since computing the Jacobian $F'(\Lambda_N,q^n)$ is
online and very expensive, whereas $F'(\Lambda_{N,0},0)$ can be calculated offline. When
solving \eqref{eqn:newton}, we stack the real and imaginary parts of the matrix
$F'(\Lambda_{N,0},0)$ and the vector $F(\Lambda_N,q^n)$ as follows
\begin{equation*}
\mathbf{J} = \left[\begin{array}{c}\Re(F'(\Lambda_{N,0},0))\\
\Im(F'(\Lambda_{N,0},0))
\end{array}\right]\quad \mbox{and}\quad
\mathbf{r}^n = \left[\begin{array}{c}\Re(F(\Lambda_{N},q^n))\\
\Im(F(\Lambda_N,q^n))
\end{array}\right].
\end{equation*}
In this situation the system has real components only which ensures that the Newton update
$\mathbf{J}^{-1}\mathbf{r}^n$ is always real. Experimentally, the frozen Newton's method
converges fairly fast. As is a general rule for applying these type of methods to
nonlinear inverse problem, the invertibility of the Jacobian is often very difficult to
establish theoretically, and this is the case here. The latter obstacle
is closely related to the lack of a uniqueness result for the inverse SLP problem.
Consequently, we are not able to establish the convergence of the quasi-Newton scheme
and, nor can we derive error estimates for the finite-dimensional approximations.
However, it is conjectured that
$F'(\Lambda_{N,0},0)$ for any $\alpha\in(1,2)$ is invertible in view of our experimental
experiences.

\subsection{Numerical results and discussions}\label{ssec:recon}

Now we are ready to present numerical reconstructions of the potentials $q_1$ and $q_2$
from finite spectral data $\Lambda_N$. For the spectral data ($\Lambda_N$) generation and
for the inversion step, we solve the initial value problem \eqref{eqn:ivp} with the
predictor-corrector algorithm (see Appendix \ref{sec:app}) with a mesh size
$h=\mbox{1e-3}$ and $h=\mbox{1.25e-3}$, respectively, in order to reduce the most obvious
form of ``inverse crime'', but not completely due to the use of identical predictor-corrector algorithm
in the forward and inverse steps. The eigenvalues are found by a quasi-Newton method in Appendix
\ref{sec:newton}. The basis set $\{w_k\}$ is fixed at $\{\sin k\pi\}$. All the
computations were performed with \textsc{matlab} R. 7.12.0(2011a) on a dual core desktop
computer.

\begin{table}[h!]
   \centering \caption{Reconstruction error $e$ for the potentials.}{\small
   \tabcolsep=5pt 
   \begin{tabular}{|c|ccccccccc|}
     \hline
     $\alpha$      &  $1.02$  & $1.1$   & $6/5$   & $4/3$   & $3/2$   & $5/3$  & $7/4$  & $9/5$  & $1.85$\\
     \hline
     $q_1$ ($N\!=\!8$)  & 4.217e-3 &4.156e-3 &4.071e-3 & 4.013e-3& 4.042e-3&4.690e-3&9.169e-3&2.576e-2&1.218e-1\\
     $q_2$ ($N\!=\!10$) & 1.983e-1 &1.983e-1 &1.983e-1 & 1.984e-1& 1.988e-1&2.026e-1&3.331e-1&1.709e0 &3.469e0 \\
     \hline
   \end{tabular}} \label{tab:err}
\end{table}

The numerical results are summarized in Table \ref{tab:err} and Figure~\ref{fig:recon}.
In the table, the error $e$ of an approximation $\tilde{q}$ to the potential $q$ is
defined by $e=\|q-\tilde{q}\|_{L^2(0,1)}$. We have experimented the reconstruction
algorithm with different numbers of basis and eigenvalues. Although the concrete numbers
differ slightly, the observations drawn from these results are the same. Hence, we shall
present only the results with $N=M=8$ and $N=M=10$ for the potentials $q_1$ and $q_2$,
respectively. We remark that the convergence of the frozen Newton method is steady and
fast, effective numerical convergence  being generally achieved within five iterations,
cf. Figure~\ref{fig:conv} for an illustration.

\begin{figure}
  \centering
  \begin{tabular}{cc}
    \includegraphics[trim = 20mm 2mm 20mm 2mm, clip, width=.5\textwidth]{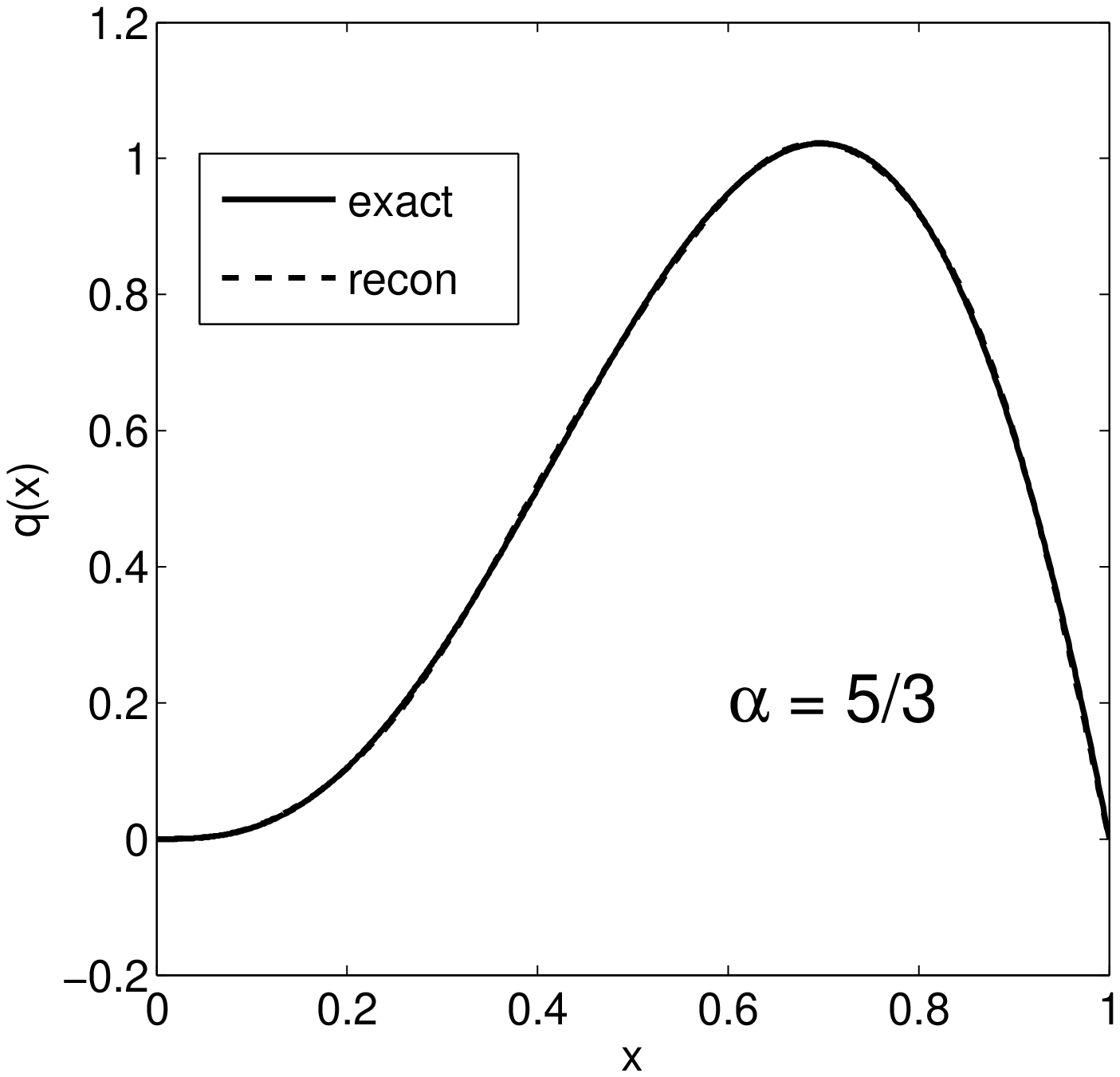} &
    \includegraphics[trim = 20mm 2mm 20mm 2mm, clip, width=.5\textwidth]{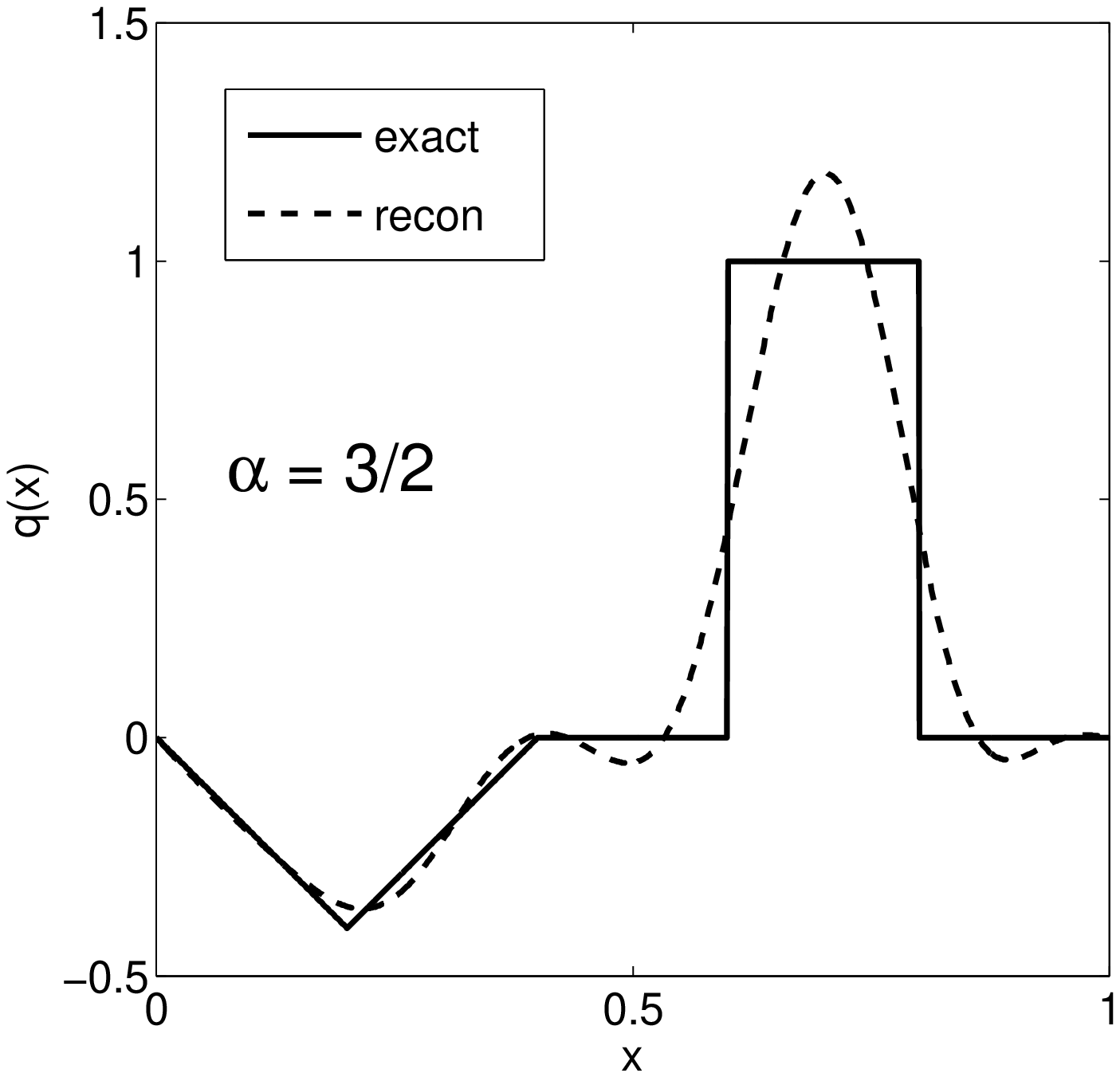}\\
    \includegraphics[trim = 20mm 2mm 20mm 2mm, clip, width=.5\textwidth]{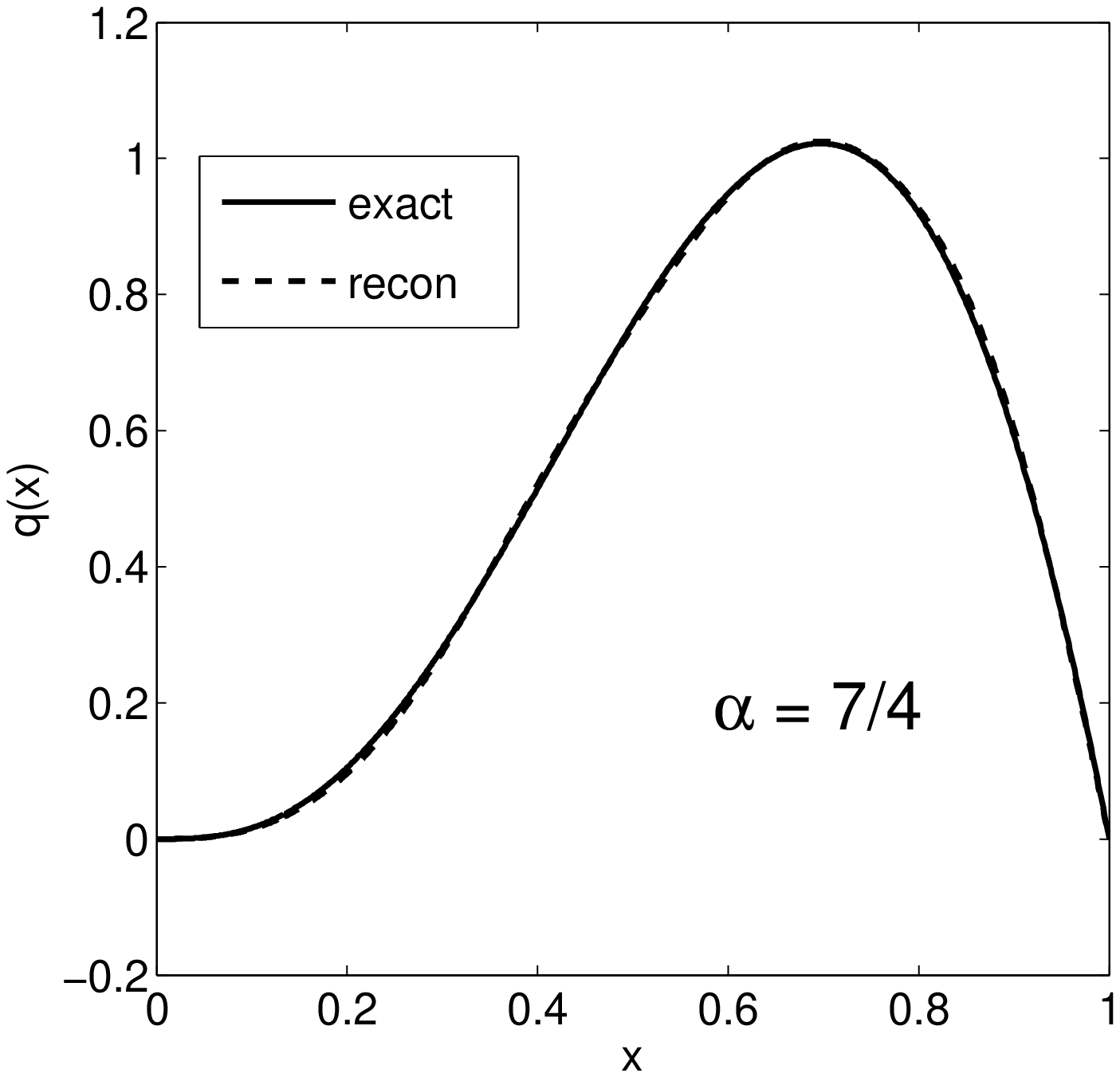} &
    \includegraphics[trim = 20mm 2mm 20mm 2mm, clip, width=.5\textwidth]{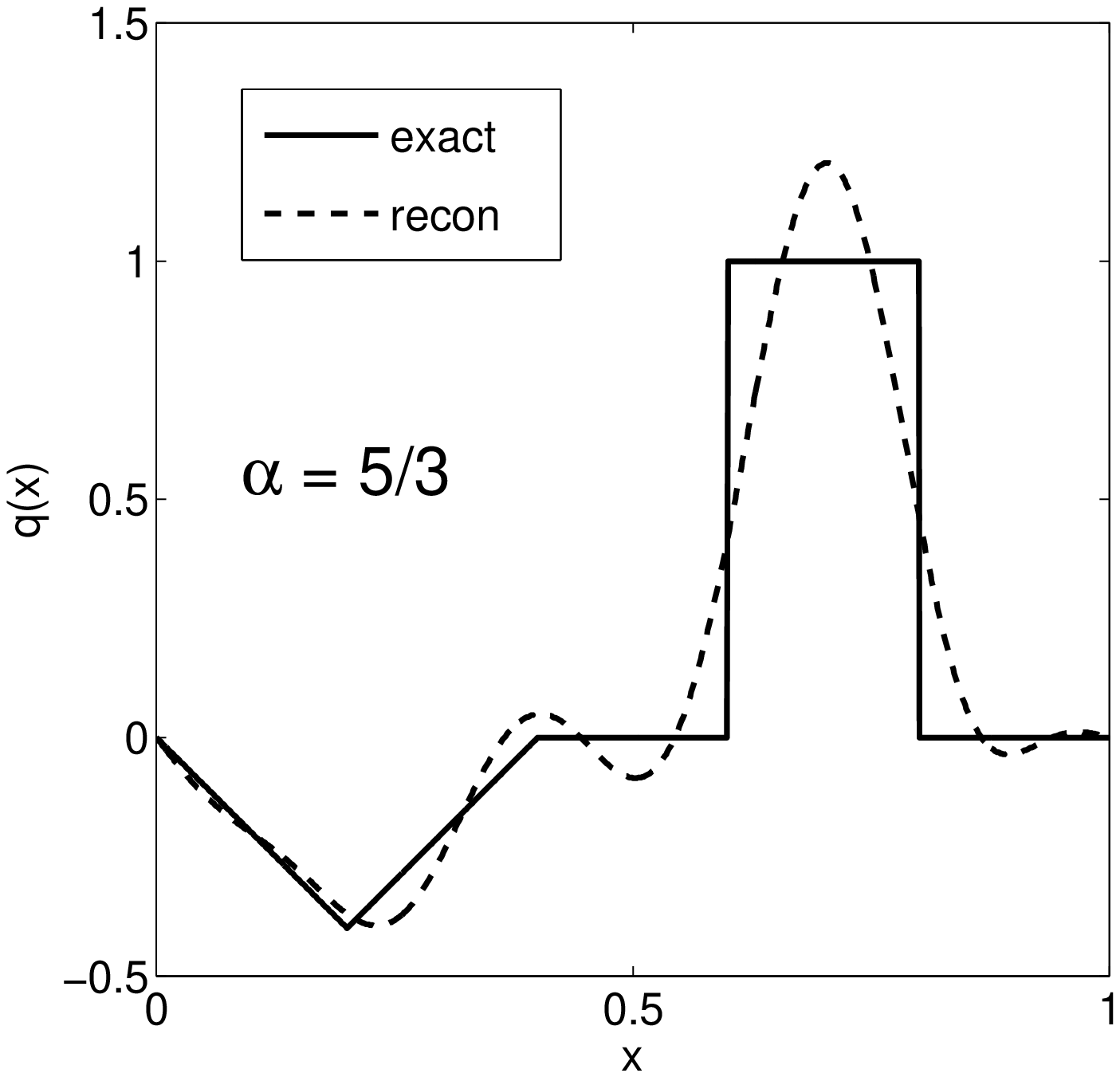}\\
    \includegraphics[trim = 20mm 2mm 20mm 2mm, clip, width=.5\textwidth]{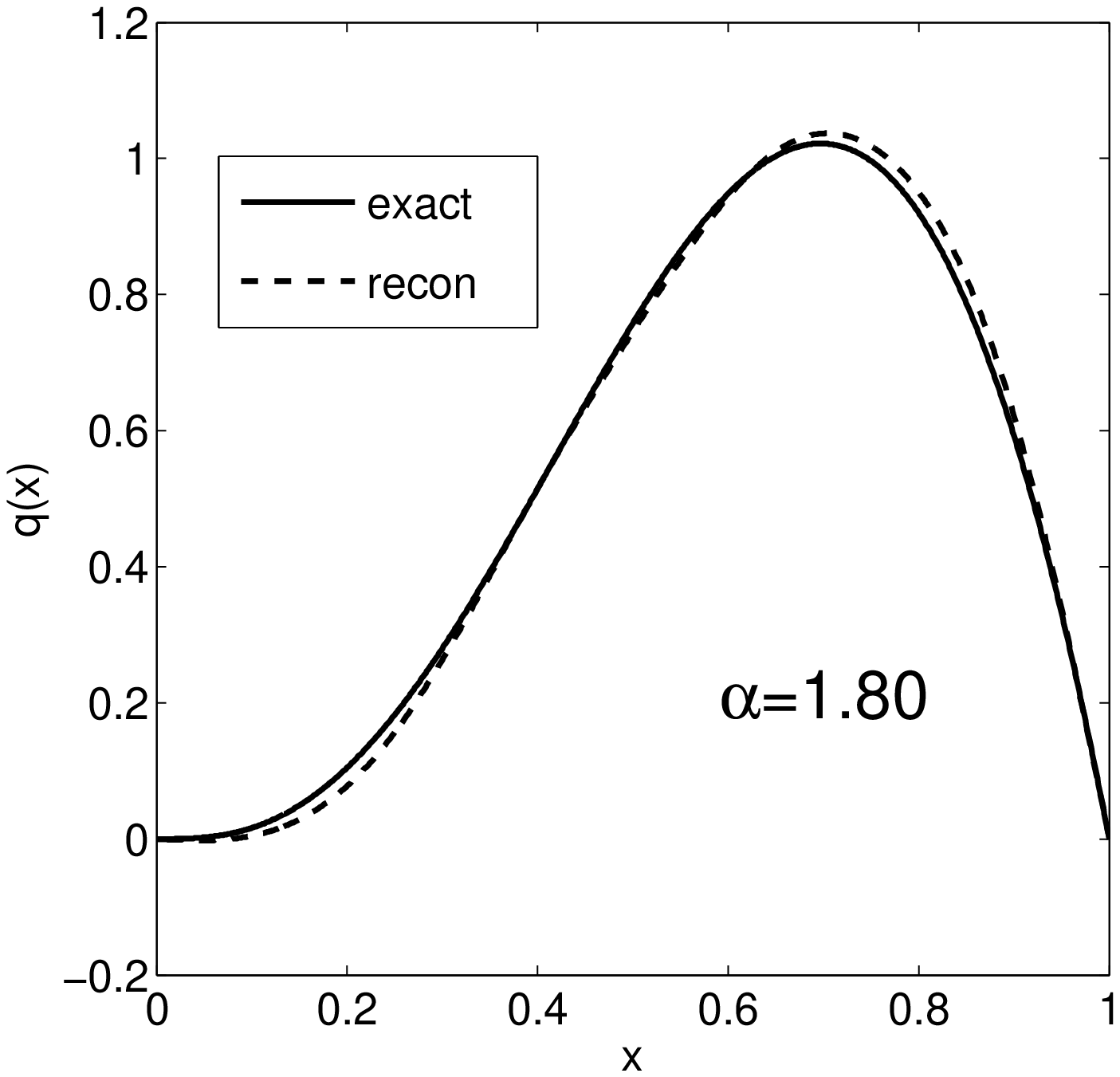} &
    \includegraphics[trim = 20mm 2mm 20mm 2mm, clip, width=.5\textwidth]{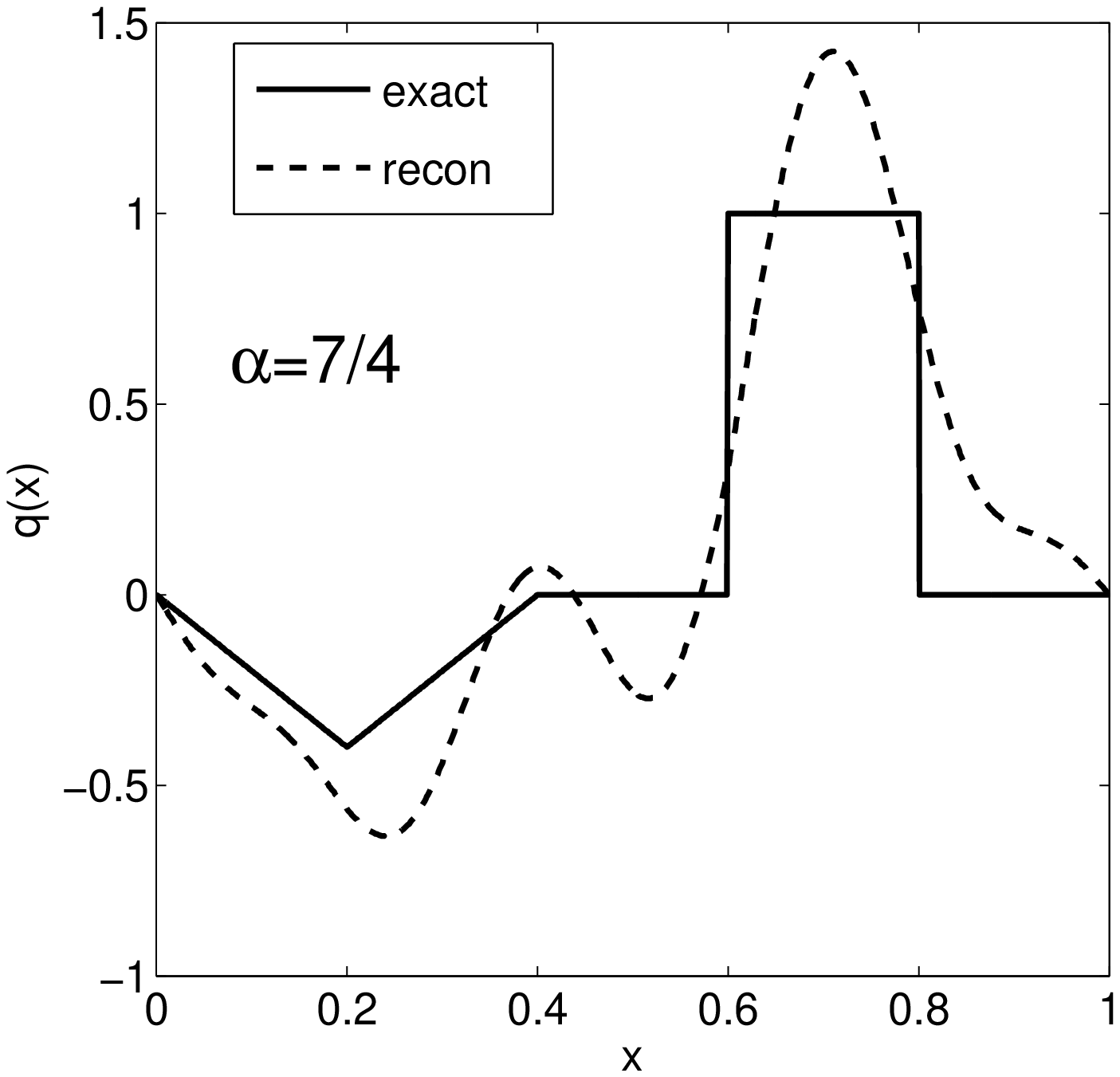}\\
  \end{tabular}
  \parbox{5in}{\caption{Numerical reconstructions for the potentials $q_1$ (left) and $q_2$ (right).}\label{fig:recon}}
\end{figure}

For $\alpha$ in the range $(1,\frac{5}{3}]$, the error $e$ of the recovered potential
$q_1$ largely stays constant, and then it deteriorates as the exponent $\alpha$ gets
closer to $2$. Some reconstructed profiles are shown in the left column of
Figure~\ref{fig:recon}. The reconstructions are in excellent agreement with the true
potential for $\alpha\in(1,\frac{5}{3}]$. Hence, one single spectrum allows accurately
reconstructing a general potential. For $\alpha$ value very close to $2$, e.g.,
$\alpha=1.9$, we are unable to obtain a good approximation by even using $15$
eigenvalues. This is expected since in the limit case $\alpha=2$, it is well known that
one single spectrum is insufficient for completely determine the potential, and only the
symmetric part can be recovered \cite{LowePilantRundell:1992}. This is corroborated by
the dramatic increase of the condition number of the Jacobian $F'(\Lambda_{N,0},0)$, cf.
Table \ref{tab:cond}: the conditioning of the stacked Jacobian $\mathbf{J}$ worsens
steadily as the exponent $\alpha$ approaches $2$ as expected since it must be singular
for $\alpha=2$. The ill-conditioning is inherently related to lack of information and is
particularly relevant for noisy data. There are several possible remedies. The most
straightforward one is to use a smaller number of basis functions which directly leads to a
better conditioned Jacobian matrix. This strategy has proved very effective in our
experiments. Another very natural approach would be to employ regularization techniques,
e.g., Tikhonov regularization and truncated singular value decomposition. Perhaps
surprisingly, these were not able to produce better results and sometimes resulted in
inferior ones.

\begin{table}
\centering \caption{Condition number of the stacked Jacobian matrix $\mathbf{J}$.}
{\small
  \begin{tabular}{|l|ccccccccc|}
  \hline
  \qquad$\alpha$     & $1.02$ & $1.1$  & $6/5$  & $4/3$  & $3/2$  & $5/3$  & $7/4$  & $9/5$  & 1.85  \\
  \hline
  $N=M=5$        &5.313e0 & 6.393e0& 8.302e0& 1.236e1& 2.095e1&2.716e2 & 2.573e3& 4.136e3& 1.193e4\\
  $N=M=8$        &8.267e0 & 1.054e1& 1.479e1& 2.439e1& 4.608e1&6.109e2 & 8.767e3& 3.397e5& 1.114e5\\
  $N=M=10$       &1.049e1 & 1.399e1& 2.082e1& 3.743e1& 7.978e1&1.043e3 & 1.110e4& 3.298e5& 3.711e5\\
  $N=M=15$       &1.656e1 & 2.411e1& 4.037e1& 8.586e1& 2.299e2&3.998e3 & 8.524e4& 1.345e6& 1.895e7\\
  \hline
  \end{tabular}\label{tab:cond}
  }
\end{table}

Similar observations can be drawn from the numerical results for the discontinuous
potential $q_2$; see the right column of Figure~\ref{fig:recon}. The reconstruction in
case of $\alpha=\frac{3}{2}$ (or in fact, any lesser value of $\alpha$) is very close to
the best Fourier approximation of the exact potential, and given our basis representation
approach, it represents an optimal reconstruction. For larger $\alpha$, and strongly so
as $\alpha$ approaches $\alpha=2$, the higher-order Fourier coefficients needed for
faithfully capturing the potential $q_2$ are less well resolved and hence the
reconstructions are less accurate. In comparison with the smooth potential, the onset of
deterioration also occurs for a smaller $\alpha$ value.

These empirical observations strongly indicate that the eigenvalues of the fractional SLP contain
more information than that of the standard second-order SLP, especially for $\alpha$ values
sufficiently away from $2$. This might be attributed to the fact that the eigenvalues of the
fractional SLP are genuinely complex, rather than real as for the classical SLP.
Similar phenomena have been observed previously in related problems in second-order SLPs and
first-order systems. In \cite{RundellSacks:2004} the authors demonstrated the unique
identifiability of the potential from one complex eigenvalue sequence for the second-order
SLP with a complex impedance boundary condition. Yamamoto
\cite{Yamamoto:1988} showed that two sequences
of (complex) eigenvalues of a certain first-order matrix operator uniquely determine two unknown
coefficients in the operator.

Lastly, we contrast the experimentally observed uniqueness for the fractional SLP with
the classical $\alpha=2$ in the linearized case. It is tempting to attribute the very
different behavior of the fractional case $1<\alpha<2$ from the classical $\alpha=2$ to
the ``anomalous diffusion'' assumption. But the behavior for $\alpha=2$ might be viewed
as the unusual circumstance. To this end, we let $q_1$ and $q_2$ be two potentials, and
$\{\lambda_n\}$ and $\{\mu_n\}$ be respective Dirichlet spectrum, and $\{u_n\}$ and
$\{v_n\}$ eigenfunctions. Then
\begin{equation*}
 \int_0^1 (q_1-q_2)u_nv_n\,dx = (\lambda_n-\mu_n)\int_0^1 u_nv_n\,dx.
\end{equation*}
If the potentials have the same spectrum, then
\begin{equation*}
\int_0^1 (q_1-q_2)u_nv_n\,dx = 0 \quad\forall n.
\end{equation*}
Now in the case $\alpha=2$, the eigenfunctions $ \{u_n\}$ behave asymptotically like
$\frac{1}{\sqrt{\lambda_n}}\sin(\sqrt{\lambda_n}x)+O(\frac{1}{\sqrt{\lambda_n}})$ and
$\sqrt{\lambda_n} = n\pi + O(\frac{1}{n})$, and similarly $\{v_n\}$ and $\mu_n$
\cite[Chap. 3]{ChadanColtonPaivrintaRundell:1997}. Meanwhile, the asymptotic expansion
implies that the mean values of the potentials $q_1$ and $q_2$ are identical. Hence, the
identity $\sin^2n\pi x=\frac{1}{2}(1-\cos2n\pi x)$ implies that to leading order, there
holds
\begin{equation*}
  \int_0^1 (q_1-q_2)\cos(2n\pi x)\,dx = 0\quad \forall n.
\end{equation*}
Therefore, to leading order, the even part of the potential difference $q_1-q_2$ is zero;
that is the Dirichlet spectrum determines the even part of the difference, but equally,
gives no information on the odd part. A refinement of this idea leads to the original
uniqueness proof due to Borg \cite{Borg:1946} of both the two spectrum problem and the
case when the potential is known to be symmetric.

However, this proof and its consequences rely on the fact that the set of squares of the
eigenfunctions $\{u_n\}$ only spans effectively ``one half'' of the space $L^2(0,1)$,
upon ignoring the small errors in the procedure, and this is due entirely to the
trigonometric identity relating the squares of eigenfunctions of order $n$ to
eigenfunctions of order $2n$. In the case $1<\alpha<2$, even for $q=0$, we have no such
identity relating $u_n^2$ to another single eigenfunction; in fact we would expect an
expression of the form $u_n^2(x) = \sum_{k=1}^\infty a_k u_k(x)$ where none of the
complex-valued sequence $\{a_k\}$ vanishes. Looked at in this light, the anomalous
behavior is in fact exhibited by the case $\alpha=2$.

\begin{figure}
  \centering
  \begin{tabular}{cc}
    \includegraphics[trim = 20mm 2mm 20mm 2mm, clip, width=.5\textwidth]{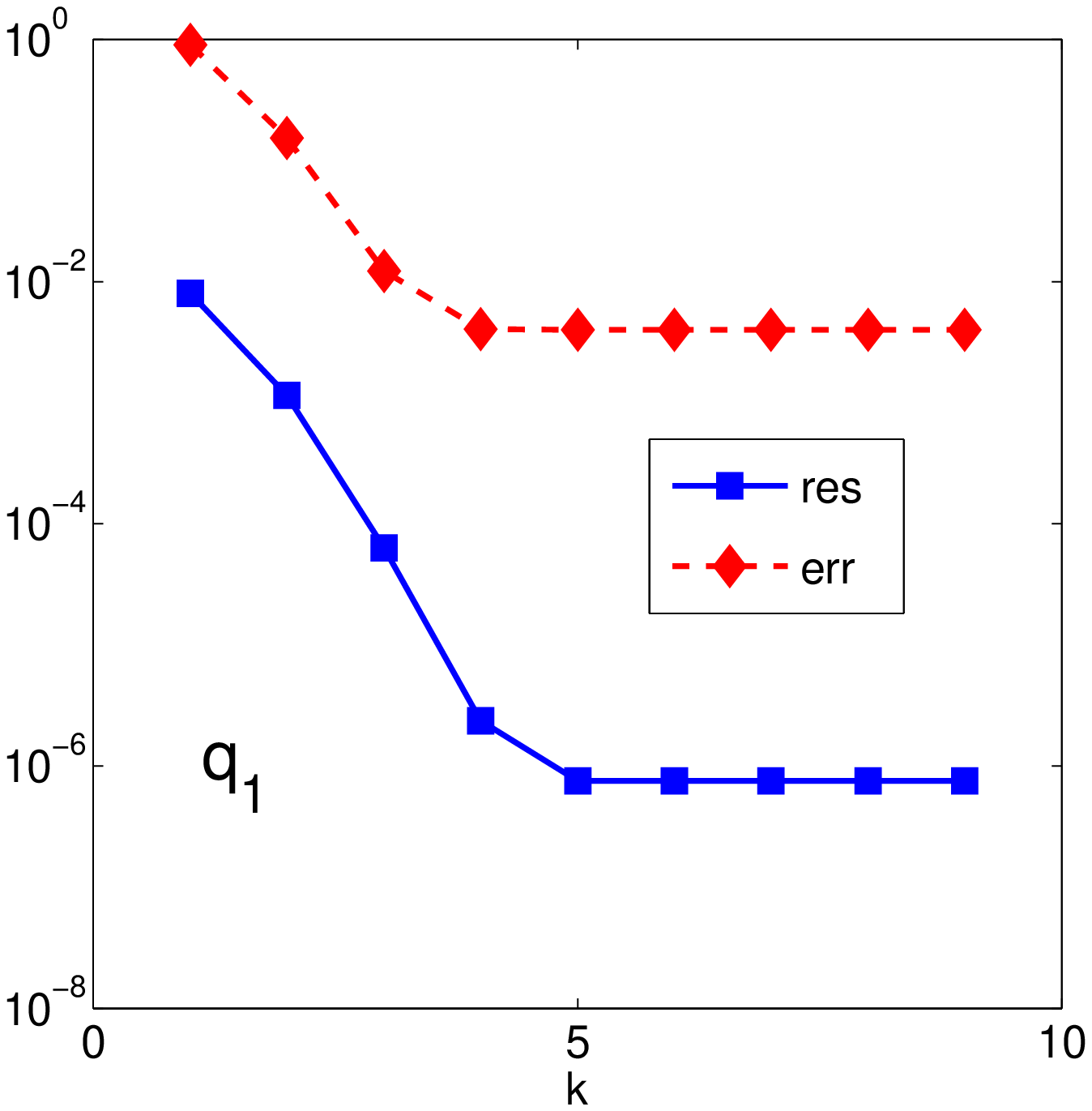} &
    \includegraphics[trim = 20mm 2mm 20mm 2mm, clip, width=.5\textwidth]{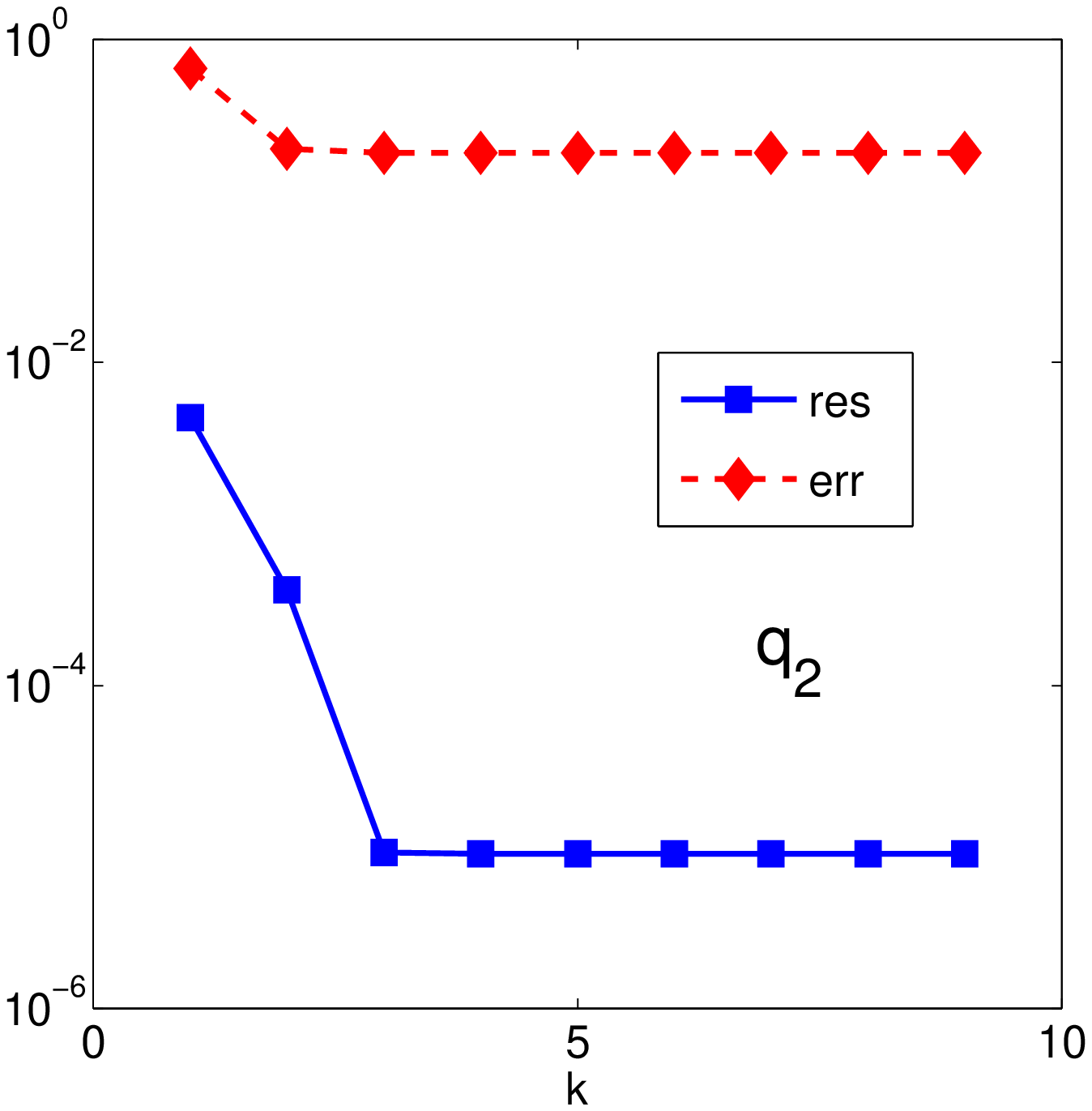}
  \end{tabular}
  \parbox{5in}{\caption{Convergence of the frozen Newton method for $\alpha=3/2$. Here $res$ and $err$ refer
  to the residual $\|F(q,\Lambda_N)\|$ and the error $e$, respectively.}\label{fig:conv}}
\end{figure}

\section{Concluding remarks}\label{sec:concl}
We have presented a first numerical study of the forward and inverse Sturm-Liouville
problems for fractional differential operators. The numerical results indicate a stunning
observation: one single spectrum can uniquely determine a general potential, and a few
(finite) spectral data allow very good reconstructions, provided that the fractional
order $\alpha\in(1,2)$ is sufficiently less than $2$.

These experimental observations lead to a number of interesting conjectures for the
fractional Sturm-Liouville problem:
\begin{itemize}
  \item[(a)] All the eigenvalues of the operator $-D_0^\alpha$ ($\alpha\in(1,2)$) are
      simple.
  \item[(b)] The (Dirichlet) eigenvalues of $-D_0^\alpha+q(x)\ (1<\alpha<2)$ obey the
      relation \eqref{eqn:decay} with $\{c_n\}\in\ell^2$.
  \item[(c)] For any $\alpha\in(1,2)$, one single (Dirichlet) spectrum completely
      determines the potential $q$ in the fractional Sturm-Liouville problem
      \eqref{eqn:slp}.
  \item[(d)] While (c) is the optimal result, a weaker version might be more
      tractable: to show the injectivity of the linearized map.
  \item[(e)] There are clear extensions to other boundary conditions and also the
      case of the half line.
\end{itemize}

Apart from these theoretical questions, the rigorous analysis of relevant numerical
schemes is also of immense interest. These include error estimates for the
eigenvalue approximations, convergence and convergence rates of the quasi-Newton
scheme etc. A better understanding of the analytical aspects of the fractional
Sturm-Liouville problem is crucial for addressing these numerical issues.

\section*{Acknowledgements}
This work is supported by Award No. KUS-C1-016-04, made by King Abdullah University of
Science and Technology (KAUST), and NSF award DMS-0715060.

\appendix
\section{Predictor-corrector algorithm}\label{sec:app}
Here we describe a predictor-corrector algorithm for fractional initial value problems
based on the idea in the work \cite{DiethelmFordFreed:2002}. There are several errors in
the formulas in \cite{DiethelmFordFreed:2002}, in, e.g. eq. (15). Hence we present
necessary details here. Consider the initial value problem (with $\alpha\in(1,2))$
\begin{equation*}\left\{
  \begin{aligned}
   & D_0^\alpha u = f(x,u(x))\quad 0<x<1,\\
   & u(0)= u_0,\quad u'(0) = u_0^{(1)}.
  \end{aligned}\right.
\end{equation*}

This initial value problem is equivalent to the Volterra integral equation
\cite[eq.~(3.1.41), pp.~141]{KilbasSrivastavaTrujillo:2006}
\begin{equation}\label{eqn:volterra}
   u(x) = u_0 + u_0^{(1)}x + \frac{1}{\Gamma(\alpha)}\int_0^x(x-t)^{\alpha-1}f(t,u(t))dt.
\end{equation}
To solve equation \eqref{eqn:volterra}, we partition the unit interval $[0,1]$ into a
uniform mesh $\mathcal{T}=\{x_k=kh, k=0,1,\ldots,K\}$, $K\in\mathbb{N}$, with a mesh size
$h=\frac{1}{K}$. The predictor-corrector algorithm in \cite{DiethelmFordFreed:2002}
extends the Adams-Bashforth-Moulton method. The predictor step applies the left rectangle
rule (with the kernel $(x_{k+1}-x)^{\alpha-1}$ being the weight) to the integral in
\eqref{eqn:volterra}, i.e.,
\begin{equation*}
  \int_0^{x_{k+1}}(x_{k+1}-x)^{\alpha-1}g(x)dx \approx \frac{h^\alpha}{\alpha}\sum_{j=0}^k b_{j,k+1}g(x_j),
\end{equation*}
where the integrand $g$ is continuous and the weights $b_{j,k+1}$ are given by
\begin{equation*}
  b_{j,k+1} = (k+1-j)^\alpha-(k-j)^\alpha,\quad j=0,1,\ldots, k.
\end{equation*}
The corrector step applies the trapezoidal rule (again with the kernel $(x_{k+1}-x)^{\alpha-1}$
being the weight) to the integral in \eqref{eqn:volterra}, i.e.,
\begin{equation*}
    \int_0^{x_{k+1}}(x_{k+1}-x)^{\alpha-1}g(x)dx \approx \frac{h^\alpha}{(\alpha+1)\alpha}\sum_{j=0}^{k+1}a_{j,k+1}g(x_j),
\end{equation*}
where the weights $a_{j,k+1}$ are given by
\begin{equation*}
  a_{j,k+1}=\left\{\begin{array}{ll}
    k^{\alpha+1}-(k-\alpha)(k+1)^\alpha, & \mbox{if } j=0,\\
    (k-j+2)^{\alpha+1}+(k-j)^{\alpha+1} - 2(k-j+1)^{\alpha+1}, &\mbox{if } j= 1,\ldots,k,\\
    1, & \mbox{if } j = k+1.
  \end{array}\right.
\end{equation*}

The predictor-corrector algorithm applies repeatedly the above two quadrature rules to get the
approximations $u_h^\mathrm{p}(x_{k+1})$ and $u_h^\mathrm{c}(x_{k+1})$ respectively by
\begin{equation*}
  \begin{aligned}
   u_h^\mathrm{p}(x_{k+1})&=u_0+u_0^{(1)}x_{k+1} + \frac{h^\alpha}{\Gamma(1+\alpha)}\sum_{j=0}^kb_{j,k+1}f(x_j,u_h(x_j)),\\
   u_h^\mathrm{c}(x_{k+1})&= u_0+u_0^{(1)}x_{k+1} + \frac{h^\alpha}{\Gamma(2+\alpha)}f(x_{k+1},u_h^\mathrm{p}(x_{k+1})
     + \frac{h^\alpha}{\Gamma(2+\alpha)}\sum_{j=0}^ka_{j,k+1}f(x_{j},u_h(x_j)).
  \end{aligned}
\end{equation*}
Then the sought-for approximation $u_h(x_{k+1})$ is obtained by re-evaluate the corrector
step with $u_h^\mathrm{c}(x_{k+1})$ in place of $u_h^\mathrm{p}(x_{k+1})$. To further
improve the accuracy of the approximation $\tilde{u}_h$, we adopt a Richardson
extrapolation step, which reads $\tilde{u}_h = \frac{1}{3}(4u_{\frac{h}{2}}-u_h)$. The
solver for the initial value problem always incorporates this extrapolation step.

\section{Quasi-Newton method for eigenvalue problem}\label{sec:newton}

Next we describe a quasi-Newton method for finding an eigenvalue $\lambda$ to the SLP
\eqref{eqn:slp}. The starting point of the method is the following obvious observation:
any eigenvalue $\lambda$ to the SLP \eqref{eqn:slp} is one root of the nonlinear map from
the potential $q$ to $u(q,\lambda)(1)$, where $u(q,\lambda)$ is the solution to the
initial value problem \eqref{eqn:ivp}. Clearly, then the solution $u(q,\lambda)$ will
also be the respective eigenfunction.

We shall find the eigenvalues to the SLP \eqref{eqn:slp} by a quasi-Newton method, the
secant method. The complete algorithm is listed in Algorithm \ref{alg:qneig}. The
stopping criterion at Step 8 can be based on monitoring the quasi-Newton update
$\delta\lambda$: if the increment $|\delta\lambda|$ falls below a given tolerance
(which is set to $1.0\times10^{-12}$ in our computation), then  the algorithm is
terminated. Accurate initial guesses for the SLP \eqref{eqn:slp} with zero potential can
be obtained from the zeros of Mittag-Leffler function $E_{\alpha,2}(-\lambda)$, which can
be directly extracted by visualizing the function $E_{\alpha,2}(-\lambda)$
\cite{SeyboldHilfer:2008}. That for a nonzero potential can be obtained by adding the
integral $\int_0^1q(t)dt$ to the eigenvalues for zero potential. In practice, including a
small imaginary part in the initial guesses is beneficial. Each iteration of the
algorithm invokes solving one initial value problem, which is dominant in the expense.
Our numerical experiences indicate that its convergence appears always superlinear, and
the convergence basin is relatively large.

\begin{algorithm}
  \caption{Quasi-Newton method for eigenvalue problem \eqref{eqn:slp}.}\label{alg:qneig}
  \begin{algorithmic}[1]
    \STATE Given two initial guesses $\lambda^0$ and $\lambda^1$;
    \STATE Calculate $u(\lambda^0,q)(1)$ and $u(\lambda^1,q)(1)$ by the
           predictor-corrector method;
    \FOR {$k=1,\dots,K$}
        \STATE Compute derivative $\partial_\lambda u(\lambda^k,q)(1)$ by finite difference
            \begin{equation*}
               \partial_\lambda u(\lambda^k,q)(1)\approx \frac{u(\lambda^{k},q)(1)-u(\lambda^{k-1},\lambda)(1)}{\lambda^{k}-\lambda^{k-1}};
            \end{equation*}
        \STATE Compute the quasi-Newton update $\delta\lambda$ by
            \begin{equation*}
               \delta\lambda = -\frac{u(\lambda^k,q)(1)}{\partial_\lambda u(\lambda^k,q)(1)};
            \end{equation*}
        \STATE Update the eigenvalue by $\lambda^{k+1}=\lambda^k+\delta\lambda$;
        \STATE Compute $u(\lambda^{k+1},q)(1)$ by predictor-corrector method;
        \STATE Check stopping criterion.
    \ENDFOR
    \STATE {\textbf{output}} approximate eigenvalue $\lambda^K$ and eigenfunction $u(\lambda^K,q)$.
  \end{algorithmic}
\end{algorithm}

\bibliographystyle{abbrv}
\bibliography{frac}
\end{document}